\documentclass[11pt]{amsart}
\usepackage{latexsym}
\usepackage{amsmath,amsthm,amssymb, amscd,color}
\usepackage{mathrsfs}
\usepackage[all]{xy}
\usepackage[all]{xy}
\usepackage{tikz}
\usepackage{tikz-cd}
\usepackage{adjustbox}
\usetikzlibrary{snakes}
\usepackage{blkarray} 
\usepackage{framed}

\newtheorem{innercustomgeneric}{\customgenericname}
\providecommand{\customgenericname}{}
\newcommand{\newcustomtheorem}[2]{%
	\newenvironment{#1}[1]
	{%
		\renewcommand\customgenericname{#2}%
		\renewcommand\theinnercustomgeneric{##1}%
		\innercustomgeneric
	}
	{\endinnercustomgeneric}
}

\newcustomtheorem{customthm}{Theorem}
\newcustomtheorem{customlemma}{Lemma}
\newcustomtheorem{customprop}{Proposition}

\theoremstyle{plain}
\newtheorem{theorem}{Theorem}[section]
\newtheorem{lemma}[theorem]{Lemma}
\newtheorem{corollary}[theorem]{Corollary}
\numberwithin{equation}{section}
\newtheorem{definition}[theorem]{Definition}

\newtheorem{proposition}[theorem]{Proposition}
\newtheorem{remark}[theorem]{Remark}

\theoremstyle{definition}
\newtheorem{example}[theorem]{Example}

\theoremstyle{remark}

\newcommand{\CC}{\mathbb{C}}
\newcommand{\QQ}{\mathbb{Q}}
\newcommand{\RR}{\mathbb{R}}
\newcommand{\ZZ}{\mathbb{Z}}
\newcommand{\NN}{\mathbb{N}}

\newcommand{\WG}{\mbox{WGr}}
\newcommand{\G}{\mbox{Gr}}
\newcommand{\WW}{\mathbb{W}}

\begin{document}

\title[Cohomologies of weighted Grassmann orbifolds]{Integral generalized equivariant cohomologies of weighted Grassmann orbifolds}

\author[K Brahma]{Koushik Brahma}
\address{Department of Mathematics, Indian Institute of Technology Madras, India}
\email{koushikbrahma95@gmail.com}

\author[S. Sarkar]{Soumen Sarkar}
\address{Department of Mathematics, Indian Institute of Technology Madras, India}
\email{soumen@iitm.ac.in}



\date{\today}
\dedicatory{}

\abstract  
We introduce a new definition of weighted Grassmann orbifolds. We study their several invariant $q$-cell structures and the orbifold singularities on these $q$-cells. We discuss when the integral cohomology of a weighted Grassmann orbifold has no $p$-torsion. We compute the equivariant $K$-theory ring of weighted Grassmann orbifolds with rational coefficients. We introduce divisive weighted Grassmann orbifolds and show that they have invariant cell structures. We calculate the equivariant cohomology ring, equivariant $K$-theory ring and equivariant cobordism ring of a divisive weighted Grassmann orbifold with integer coefficients. We discuss how to compute the weighted structure constants for the integral equivariant cohomology ring of a divisive weighted Grassmann orbifold.
\endabstract

\maketitle


\section{Introduction}
We consider the $n$-dimensional complex vector space $\CC^n$ and a positive integer $d$ satisfying $1\leq d < n$. Then the set of all $d$-dimensional vector subspaces of $\CC^n$ is called a (complex) Grassmann manifold and denoted by $\G(d, n)$. In particular, the space $\G (1, n)$ is called the $(n-1)$-dimensional complex projective space. The space $\G(d, n)$ has a manifold structure of dimension ${d(n-d)}$, see \cite[Chapter 1]{Muk}. This is a projective variety via the Pl\"ucker embedding. The natural $(\CC^*)^n$-action on $\CC^n$ induces a $(\CC^*)^n$-action on $\G(d, n)$. Grassmann manifolds are central objects of study in algebraic geometry, algebraic topology and differential geometry. Several interesting topological and geometrical properties of Grassmann manifolds can be found in \cite{Lak, KnTa, JiPe}. 

The orbifold version of a complex projective space was introduced in \cite{Ka} and was called a twisted projective space. Orbifolds, a generalization of manifolds, were introduced by Satake \cite{Sat2,Sa} with the name $V$-manifolds. Later, Thurston \cite{Thu} used the terminology orbifolds instead. In the past two decades, several development have been appeared to study orbifolds arising in algebraic geometry, differential geometry and string topology. Some cohomology theories such as the de Rham cohomology \cite[Chapter 2]{ALR}, the singular cohomology \cite{Hat}, the Dolbeault cohomology \cite{Bai}, Chen-Ruan cohomology ring \cite{CR} and orbifold $K$-theory \cite[Chapter 3]{ALR} for a class of orbifolds were studied either with rational, real or, complex coefficients. One can construct a CW-complex structure on an effective orbifold following \cite{Gor}. However, in general, the computation of the singular integral cohomology of an orbifold is considerably difficult.

Let $G$ be a topological group and $X$ a $G$-space. Then the equivariant map $X \longrightarrow \{pt\}$ induces a graded $\mathcal{E}_{G}^*(pt)$-algebra structure on $\mathcal{E}_{G}^*(X)$. The readers are referred to \cite{May} for
the definitions and several results on the $G$-equivariant generalized cohomology theory $\mathcal{E}_{G}^*$. If $\mathcal{E}_{G}^*=H_G^*$, then it is known as the equivariant cohomology theory defined by $$H_{G}^*(X):=H^*(EG\times_GX).$$ The ring $H_G^{*}(X)$ is called the Borel equivariant cohomology of $X$. If  $\mathcal{E}_{G}^*=K_G^*$, then it is known as the equivariant $K$-theory. If $X$ is compact then $K_G^0(X)$ is the equivalence classes of $G$-equivariant complex vector bundles on $X$ \cite{Segal}. If $X$ is a point with trivial action, then $K_G^*(pt)$ is isomorphic to $R(G)[z,z^{-1}]$ where $R(G)$ is complex representation ring of $G$ and $z$ is the Bott element of cohomological dimension $-2$. The $G$-equivariant ring $MU_G^*(X)$ is known as equivariant complex cobordism ring see \cite{TD}. Sinha \cite{Sinha} and Hanke \cite{Han} have shown several development on $MU_G^*$. However, many interesting questions on $MU_G^*(X)$ remain undetermined. For example, $MU_G^*(pt)$ is not completely known for non-trivial groups $G$.

 Corti and Reid \cite{Core} introduced the weighted projective analogs of Grassmann manifolds and called them weighted Grassmannians. Then Abe and Matsumura \cite{AbMa1} defined them explicitly and studied the equivariant cohomology ring of weighted Grassmannians with rational coefficients. The weighted Grassmannians are projective varieties with orbifold singularities. The simplest weighted Grassmannians are the weighted projective spaces. Kawasaki \cite{Ka} proved that the integral cohomology of weighted projective spaces have no torsion and is concentrated in even degrees. The equivariant cohomology ring of a weighted projective space has been studied in \cite{BFR} in terms of piecewise polynomials. The equivariant $K$-theory and equivariant cobordism rings of divisive weighted projective spaces have been discussed in \cite{HHRW} in terms of piecewise Laurent polynomials and piecewise cobordism forms respectively.

  Inspired by the above works, we introduce a different definition of weighted Grassmann orbifolds and study their several topological properties such as torsion in the integral cohomology, equivariant cohomology ring, equivariant $K$-theory ring and equivariant cobordism ring with integer coefficients. We note that \cite{Core, AbMa1}  used the name `weighted Grassmannians'. However, keeping other naming in mind like Milnor manifolds and Seifert manifolds, we prefer to use Grassmann manifolds and weighted Grassmann orbifolds.

The paper is organized as follows. In Section \ref{sec_weighted_grasm}, analogously to the definition of Grassmann manifold discussed in \cite{Muk}, we introduce another definition of a weighted Grassmann orbifold $\WG (d,n)$ for $d < n$, a `weight vector' $W:=(w_1, \ldots, w_n) \in (\ZZ_{\geq 0})^n$ and $a \in \ZZ_{\geq 1}$. Interestingly, this definition is equivalent to the previous one appeared in \cite{AbMa1}.  We recall the definition of Schubert symbols for $d<n$ and discuss how to get a total ordering on the Schubert symbols. Using this total order we show that there is a `weighted Pl\"{u}cker embedding' from a weighted Grassmann orbifold to a weighted projective space 
see Lemma \ref{lem_pul_emb}. We describe a $q$-cell structure of $\WG (d,n)$ in Proposition \ref{lem_q_cell_grasm}.
Then we discuss a $(\CC^*)^n$-invariant filtration 
 \begin{equation*}
 	\{pt\}=X_0 \subset X_1 \subset X_2 \subset \dots \subset X_m=\WG(d,n)
 \end{equation*} 
 of $\WG (d,n)$  using the $q$-cell decomposition, where $m:={n \choose d}-1$. Here, we consider $q$-cell structure in the sense of \cite[Section 4]{PS}. We note that one may get different $q$-cell structures depending on the choice of the total orderings on the set of all Schubert symbols for $d < n$. Accordingly, one may obtain different $(\CC^*)^n$-invariant filtration of $\WG (d,n)$.

In Section \ref{sec_int_cohom_grassmann}, first we recall that there is an equivariant homeomorphism from $\mathbb{W} P(rc_0, rc_1, \ldots, rc_m)$ to $\mathbb{W} P(c_0, c_1, \ldots, c_m)$ for any $1 \leq r\in \NN$. Using this technique, we show how the orbifold singularity on a $q$-cell of some subcomplexes of $\WG (d,n)$ can be reduced, see Lemma \ref{thm_red_loc_gp}. Consequently, we get a new $q$-cell structure of these subcomplexes including $\WG (d,n)$ possibly with less singularity on each $q$-cell, see Theorem \ref{prop_cell_after_red}. We show in Theorem \ref{lem_euiv_sn} that two weighted Grassmann orbifolds are weakly equivariantly homeomorphic if their weight vectors differ by a permutation $\sigma\in S_n$. We define `admissible permutation' $\sigma\in S_n$ for a prime $p$ and $\WG(d,n)$, see Definition \ref{def_adm_per}. The following result says when $H^{*}({\rm WGr} (d,n);\mathbb{Z})$ has no $p$-torsion. 

\begin{customthm}{A}[Theorem \ref{p torsion in permuted Grassmann}]
If there exists an admissible permutation ${\sigma\in S_n}$ for a prime $p$ and ${\rm WGr}(d,n)$ then $H^{odd}({\rm WGr} (d,n);\mathbb{Z}_p)$ is trivial and $H^{*}({\rm WGr} (d,n);\mathbb{Z})$ has no $p$-torsion.
\end{customthm}

We introduce `divisive'  weighted Grassmann orbifolds. We note that this definition coincides with the concept of divisive weighted projective space of \cite{HHRW} when $1=d<n$. We prove the following.

 \begin{customthm}{B}[Theorem \ref{thm:tinv_cell}]
 	If ${\rm WGr}(d,n)$ is a divisive weighted Grassmann orbifold then it has a $(\CC^*)^n$-invariant cell structure with only even dimensional cells. Moreover, the $(\CC^*)^n$-action on these cells can be described explicitly.
 \end{customthm}

 This result implies that the integral cohomology of a divisive weighted Grassmann orbifold has no torsion and is concentrated in even degrees. We discuss a class of non-trivial  examples of divisive weighted Grassmann orbifolds. We remark that the weighted Grassmann orbifold in Example \ref{ex:no_torsion2} is not divisive. However, its integral cohomology has no torsion.

In Section \ref{sec_kth_cobth}, we show that the $(\CC^*)^n$-invariant stratification $$\{pt\}= X_0 \subset X_1 \subset \cdots \subset X_m = \WG(d,n)$$ has the following property. The quotient $X_i/X_{i-1}$ is homeomorphic to the Thom space of an orbifold $(\CC^*)^n$-bundle $$\xi^i \colon \CC^{\ell(\lambda^i)}/G_i \to \{pt\}$$ for some $\ell(\lambda^i)\in(\ZZ_{\geq 1})$ and finite groups $G_i$ for $i=1, \ldots, m$, see Proposition \ref{thm:q-cw_srtucture}. We compute the equivariant $K$-theory ring of any weighted Grassmann orbifolds with rational coefficients, see Theorem \ref{prop_GKM-description}. If $\WG(d,n)$ is divisive  then $G_i$ is trivial for $i=1, \ldots, m$. Then considering $T^n:=(S^1)^n\subset (\CC^*)^n$, the following result describes the integral equivariant cohomology of certain weighted Grassmann orbifolds.
\begin{customthm}{C}[Theorem \ref{prop_GKM-description_Z}]
	Let ${\rm WGr} (d,n)$ be a divisive weighted Grassmann orbifold for $d <n$. Then the generalized $T^n$-equivariant cohomology with integer coefficient $\mathcal{E}^\ast_{T^n}({\rm WGr}(d,n); \ZZ)$ can be given by
	\begin{equation*}
		 \Big \{(f_i) \in \bigoplus_{i=0}^{m} \mathcal{E}^*_{T^n}(pt; \ZZ) ~ \big{|} ~ e_{T^n}(\xi^{ij})~
		\mbox{divides} ~ f_i - f_j ~\mbox{for }~ j < i~\mbox{and}~  |\lambda^j\cap \lambda^i|=d-1\Big\}
	\end{equation*}
for $\mathcal{E}^\ast_{T^n}= H^\ast_{T^n}, K^*_{T^n}$ and $MU^*_{T^n}$.
	\end{customthm}

The computation of $e_{T^n}(\xi^{ij})$ is discussed in \eqref{eq_eu_cl}. We compute the  equivariant cohomology ring of some weighted Grassmann orbifold with integer coefficients which are not divisive, see Theorem \ref{th_eq_coh_non_div}. For $m\geq 2$, corresponding to each pair of positive integers $(n,d)$ such that $d<n$ and $m+1={n \choose d}$ we have a $T^n$-action on $\WW P(c_0,c_1,\dots,c_m)$. For each pair $(n,d)$, we discuss the generalized $T^n$-equivariant cohomology of a divisive $\WW P(c_0,c_1,\dots,c_m)$ with integer coefficients, see Theorem \ref{thm_eq_coh_pr_sp}.

In Section \ref{sec:schub_calc}, we show that there exist equivariant Schubert classes $\{w\widetilde{S}_{\lambda^i}\}_{i=0}^m$ which form a basis for the integral $T^n$-equivariant cohomology of a divisive weighted Grassmann orbifold, see Proposition \ref{prop_schub_bas}. We study some properties of weighted structure constant, see Lemma \ref{lem_st_con}. Then we show the following multiplication rule.

\begin{customprop}{D}[Proposition \ref{prop_pieri_rull}][Weighted Pieri rule]
	$$w\widetilde{S}_{\lambda^1}~w\widetilde{S}_{\lambda^j}=({w\widetilde{S}_{\lambda^1}}|_{\lambda^j})~w\widetilde{S}_{\lambda^j}+\sum_{\lambda^i\to\lambda^j}\dfrac{c_0}{c_j}~w\widetilde{S}_{\lambda^i}.$$
\end{customprop}
 \noindent Moreover, we deduce a recurrence relation which helps to compute the weighted structure constants $\{wc_{ij}^k\}$ corresponding to this Schubert basis $\{w\widetilde{S}_{\lambda^i}\}_{i=0}^m$ with integral coefficients.

 \begin{customprop}{E}[ Proposition \ref{prop_pieri_rull2}]
 	For any three Schubert symbols $\lambda^i,\lambda^j$ and $\lambda^k$, we have the following recurrence relation. $$({w\widetilde{S}_{\lambda^1}}|_{\lambda^k}-{w\widetilde{S}_{\lambda^1}}|_{\lambda^i})~wc_{ij}^{k}=(\sum _{\lambda^s\to\lambda^i}\dfrac{c_0}{c_i}~wc_{sj}^{k}-\sum _{\lambda^k\to \lambda^t}\dfrac{c_0}{c_t}~wc_{ij}^{t}).$$
 \end{customprop}
 

\section{Invariant $q$-cell structure on weighted Grassmann orbifolds}\label{sec_weighted_grasm}
In this section, we introduce another definition of weighted Grassmann orbifold $\WG (d,n)$ where $d<n$. We recall the definition of a Schubert symbol for ${d<n}$ and discuss some (total) ordering on the set of Schubert symbols. We show that there is an equivariant embedding from a weighted Grassmann orbifold to a weighted projective space. We show that our definition of weighted Grassmann orbifold is equivalent to the previous one appeared in \cite{AbMa1}. We study the orbifold and {$q$-cell} structures of weighted Grassmann orbifolds generalizing the manifolds counter part discussed in \cite{MiSt}.

A Schubert symbol $\lambda$ for $d < n$ is a sequence of $d$ integers $(\lambda_1,\lambda_2,\dots, \lambda_d)$ such that $1\leq \lambda_1 < \lambda_2 < \cdots < \lambda_d\leq n$. The length $\ell(\lambda)$ of a Schubert symbol $\lambda :=(\lambda_1, \lambda_2, \dots , \lambda_d)$ is defined by $\ell(\lambda):=(\lambda_1-1)+(\lambda_2-2)+ \cdots + (\lambda_d-d)$. There are ${n \choose d}$ many Schubert symbols for $d < n$. One can define a partial order `$\preceq$' on the Schubert symbols for $d<n$ by
\begin{equation}\label{eq_bru_ord}
 \lambda \preceq \mu \text{ if } \lambda_i\leq \mu_i \text{ for all } i=1,2,\dots,d.
\end{equation}
  Then the set of all Schubert symbols for $d<n$ form a poset with respect to this partial order `$\preceq$'.

Let $M_d(n,d)$ be the set of all complex $n \times d$ matrix of rank $d$ and $\mbox{GL}(d,\CC)$ the set of all non-singular complex matrix of order $d$.  
We denote a matrix  $A \in M_d(n,d)$ as follows $$A=
\begin{pmatrix}
a_{11} & a_{12} & \cdots & a_{1d} \\
a_{21} & a_{22} & \cdots & a_{2d} \\
\vdots & \vdots & \vdots   & \vdots  \\
a_{n1} & a_{n2} & \cdots & a_{nd}  
\end{pmatrix}= \begin{pmatrix}
{\bf a}_1  \\
{\bf a}_2  \\
\vdots    \\
{\bf a}_n 
\end{pmatrix}$$ 
where ${\bf a}_i \in \CC^d$ for $i=1, \ldots, n$.

\begin{definition}\label{def_weighted_gras}
Let $W :=(w_1,w_2, \dots ,w_n) \in (\mathbb{Z}_{\geq 0})^{n}$ and $a \in\mathbb{Z}_{\geq 1}$. We define an equivalence relation $\sim_w$ on $M_d(n,d)$ by 
$$\begin{pmatrix}
	\bf{a}_1  \\
	\bf{a}_2  \\
	\vdots    \\
	\bf{a}_n  
\end{pmatrix} \sim_w \begin{pmatrix}
 t^{w_1} \bf{a}_1 \\
 t^{w_2}  \bf{a}_2\\
\vdots     \\
t^{w_n} \bf{a}_n 
\end{pmatrix}T $$
for $T \in {\rm GL}(d,\mathbb{C})$ and $t \in \mathbb{C}^{*}$ such that $t^a=\det(T) \in \mathbb{C}^{*}$. We denote the identification space by $${\rm WGr} (d,n):=\dfrac{M_d(n,d)}{\sim_w}. $$
The quotient map 
\begin{equation}\label{eq_piw}
\pi_w \colon M_d(n,d)\to {\rm WGr} (d,n)
\end{equation}
 is defined by $\pi_w(A)=[A]_{\sim_w}$. The topology on  ${\rm WGr}(d,n)$ is given by the quotient topology via the map $\pi_w$.   
\end{definition}

\begin{remark}
   If $W=(0,0, \dots ,0)$ and $a=1$ then ${\rm WGr} (d,n)$ is the Grassmann manifold ${\rm Gr}(d,n)$. We denote the corresponding quotient map by 
   \begin{equation}\label{eq_map_pi}
   \pi \colon M_d(n,d)\to {\rm Gr}(d,n).
   \end{equation}
   The space ${\rm Gr}(d,n)$ is a $d(n-d)$-dimensional smooth manifold and  represents the set of all $d$-dimensional vector subspaces in $\CC^n$. Several basic properties such as manifold and a cell structure of ${\rm Gr}(d, n)$ can be found in \cite{MiSt}. 
\end{remark}

\begin{remark}\label{rmk_wgt_pr_sp}
	If $d=1$ then $M_d(n,d)=M_1(n,1)=\CC^n-\{0\}$ and ${\rm GL}(1,\CC)=\CC^*$. The corresponding $\sim_w$ is given by 
	$$(z_1,z_2,\dots,z_n)\sim_w(t^{a+w_1}z_1,t^{a+w_2}z_2,\dots,t^{a+w_n}z_n).$$ The quotient space $\frac{M_1(n,1)}{\sim_w}$ is called the weighted projective space with weights $(a+w_1, a+w_2, \dots,a+w_n)$ and denoted by $\WW P(c_0, c_1, \dots,c_{n-1})$ where $c_i=a+w_{i+1}$ for $i\in \{0,1,\dots,n-1\}$. For the weighted projective space, we denote $\sim_w$ by $\sim_c$ when $c=(c_0,c_1,\dots,c_{n-1})$. This identification $\sim_c$ is called a weighted $\CC^{*}$-action on $\CC^n-\{0\}$ with weights $(c_0, c_1, \dots,c_{n-1})$. In addition, if $W=(0,0, \dots ,0)$ and $a=1$ then $c_0=1=c_1=\dots=c_{n-1}$ and $\WW P(c_0, c_1, \dots,c_{n-1})$ is $\CC P^{n-1}={\rm Gr} (1,n)$. 
\end{remark}

\begin{definition}\label{def:dict_order}
Let $\lambda=(\lambda_1,\lambda_2,\dots,\lambda_d)$ and $\mu=(\mu_1,\mu_2,\dots,\mu_d)$ be two Schubert symbols for $d < n$. We say that $\lambda < \mu$ if $\ell(\lambda) < \ell(\mu)$, otherwise we use the dictionary order if $\ell(\lambda)=\ell(\mu)$. 
\end{definition}
This gives a total order on the set of all Schubert symbols. Note that the total order `$<$' in Definition \ref{def:dict_order} preserves the partial order `$\preceq$' in \eqref{eq_bru_ord}. That is, for two Schubert symbols $\lambda$ and $\mu$, $\lambda\preceq \mu \implies \lambda \leq \mu$, but the converse may not be true in general. Observe that there may exist several other total orders on the set of all Schubert symbols which preserve the partial order `$\preceq$'. For example, the dictionary order also gives a total order on the Schubert symbols. By a total order on the set of all Schubert symbols for $d<n$, we mean one of these total orders on it.  For $m := {n \choose d} -1$, let
\begin{equation}\label{eq_total_ord}
\lambda^0<\lambda^1<\lambda^2< \dots <\lambda^m
\end{equation}
 be a total order on the Schubert symbols for $d <n$.


For $W=(w_1,w_2, \dots ,w_n)\in(\mathbb{Z}_{\geq 0})^{n}$, $a \in\mathbb{Z}_{\geq 1}$ and $i\in \{0,1,\dots,m\}$, let  
\begin{equation}\label{eq_wli}
c_i:=a+\sum_{j=1}^{d} w_{\lambda_j^i}
\end{equation}
 where $\lambda^i = (\lambda_1^i,\lambda_2^i,\dots,\lambda_d^i)$ is the $i$-th Schubert symbol given in \eqref{eq_total_ord}. Then $c_i \geq 1$ for any $i \in \{0, \ldots, m\}$. Therefore, one can define the weighted projective space $\WW P(c_0, c_1, \dots, c_m)$ from Remark \ref{rmk_wgt_pr_sp}. We denote the associated orbit map $\mathbb{C}^{m+1}-\{0\} \to \WW P(c_0, c_1, \dots, c_m)  $ by $\pi'_c$ which can be written as
 \begin{equation}\label{eq_pic}
 	\pi'_c(z_0,z_1, \ldots, z_m) = [z_0:z_1: \cdots : z_m]_{\sim_c}.
 \end{equation} 
 Note that when $c_0=c_1= \cdots =c_m=1$, then the corresponding orbit map is denoted by $$\pi' \colon \mathbb{C}^{m+1}-\{0\} \to \CC P^m.$$

Let $(t_1,t_2,\dots,t_n) \in (\CC^{*})^n$ and $A=({\bf a}_1,{\bf a}_2,\dots ,{\bf a}_n)^{tr} \in M_d(n,d)$. Then $(\CC^{*})^n$  acts on $M_d(n,d)$ defined by  
\begin{equation}\label{T_act_Gsm}
	(t_1, \ldots, t_n)({\bf a}_1,{\bf a}_2,\dots ,{\bf a}_n)^{tr}:=(t_1{\bf a}_1,t_2{\bf a}_2, \dots ,t_n{\bf a}_n)^{tr}.
\end{equation}
This induces a natural $(\CC^{*})^n$-action on $\WG (d,n)$ such that the orbit map $\pi_w$ of \eqref{eq_piw} is $(\CC^{*})^n$-equivariant.

We remark that the standard ordered basis $\{e_1, e_2, \dots, e_n\}$ of $\mathbb{C}^n$ induces an ordered basis $\{e_{\lambda^0}, e_{\lambda^1}, \dots, e_{\lambda^m}\}$ of $\Lambda^d(\mathbb{C}^n)$. Therefore, we can identify $\Lambda^d(\mathbb{C}^n)$ with $\mathbb{C}^{m+1}(=\CC \{e_{\lambda^0}, e_{\lambda^1}, \dots, e_{\lambda^m}\})$.
The standard action of $(\CC^{*})^n$ on $\mathbb{C}^n$ induces an action of $(\CC^{*})^n$ on $\CC^{m+1}-\{0\}$ which is defined by
\begin{equation}\label{T_act}
	(t_1,t_2,\dots ,t_n)(\sum_{i=0}^{m} a_{i}e_{\lambda^i})=\sum_{i=0}^{m} a_{i} t_{\lambda^i} e_{\lambda^i}
\end{equation}
where $t_{\lambda}=t_{\lambda_1}\cdots t_{\lambda_{d}}$ and $e_{\lambda}=e_{\lambda_1} \wedge \cdots \wedge e_{\lambda_{d}}$ for the Schubert symbol $\lambda =(\lambda_1,\lambda_2,\dots,\lambda_d)$.
This induces a $(\CC^{*})^n$-action on the weighted projective space $\WW P(c_0, c_1, \ldots, c_m)$ such that the orbit map $\pi^{'}_c$ in \eqref{eq_pic} is $(\CC^{*})^n$-equivariant.

 For each Schubert symbol $\lambda =(\lambda_1,\lambda_2,\dots,\lambda_d)$, let $A_\lambda$ be the matrix with row vectors $ {\bf a}_{\lambda_1}, {\bf a}_{\lambda_2}, \dots, {\bf a}_{\lambda_d}$. Consider the map $P\colon M_d(n,d)\to \CC^{m+1}$  defined by
\begin{equation}\label{embedding of matrix}
	P(A) = {\bf v}_1 \wedge {\bf v}_2 \wedge \dots \wedge { \bf v}_d =\sum_{i=0}^m \det(A_{\lambda^i}) e_{\lambda^i},
\end{equation}
where ${\bf v}_1,{\bf v}_2, \ldots, {\bf v}_d \in \CC^n$ are the columns of $A$. Observe that $ P(A)\neq 0$ because $A\in M_d(n,d)$ has rank $d$.

\begin{lemma}\label{lem_pul_emb}
	The map in \eqref{embedding of matrix} induces a weighted Pl\"{u}cker embedding $$Pl_w \colon {\rm WGr}(d,n) \rightarrow \WW P(c_0, c_1, c_2,\dots, c_m).$$
\end{lemma}

\begin{proof}
From \eqref{embedding of matrix} we have 
$$P(D A T) = \sum_{i=0}^{m} \det((DAT)_{\lambda^i}) e_{\lambda^i}= \sum_{i=0}^m t^{c_i}\det(A_{\lambda^i}) e_{\lambda^i},$$
 where $T \in \mbox{GL}(d,\mathbb{C})$, $D =\mbox{diag}(t^{w_1}, t^{w_2}, \ldots,  t^{w_n})$ is the diagonal matrix for $t \in \mathbb{C}^{*}$ such that $t^a=\det(T)$ and $c_i$ is defined in \eqref{eq_wli} for $i=0,1,2,\dots,m$. Therefore, this induces a map 
$$Pl_w \colon \WG (d,n) \rightarrow \WW P(c_0, c_1, \ldots, c_m)$$ 
defined by
\begin{equation}\label{weighted Pluecker map}
Pl_w([A]_{\sim_w}) = [\det(A_{\lambda^0}): \det(A_{\lambda^1}): \dots: \det(A_{\lambda^m})]_{\sim_c}.
\end{equation} 
This map satisfies the following commutative diagram 
\[ \begin{tikzcd}
	M_d(n,d) \arrow{r}{P} \arrow{d}{\pi_w} & \mathbb{C}^{m+1}-\{0\} \arrow{d}{\pi_c^{'}} \\%
		\WG(d,n) \arrow{r}{Pl_w} & \WW P(c_0,c_1,\dots,c_m).
\end{tikzcd}
\]
Thus the map $Pl_w$ is continuous, since $\pi_{w}$ and $\pi^{'}_c$ are quotient maps.

 Let $[A]_{\sim_w} \in \WG(d,n)$ for some $A\in M_d(n,d)$. So, there exists a Schubert symbol $\lambda^i$ such that $\det(A_{\lambda^i})\neq 0$. Without loss of generality, we can assume that $A_{\lambda^i}=\mbox{I}_{d}$, where $\mbox{I}_{d}$ is the identity matrix of order $d$. If $A_{\lambda^i}\neq \text{I}_d$ then one can calculate $s\in \CC^{*}$ such that $s^{c_i}=1/\det(A_{\lambda^i})$. Now we consider the matrices  $D=\mbox{diag}(s^{w_1},s^{w_2}, \dots,s^{w_n})$ and $T=(D_{\lambda^i}A_{\lambda^i})^{-1}$. Then $\det(T)=s^a$ and $(DAT)_{\lambda^i}=\text{I}_d$. Note $[DAT]_{\sim_w}=[A]_{\sim_w} \in \WG(d,n)$.

We prove that $Pl_w$ is injective. Let $[A]_{\sim_w}$ and $[B]_{\sim_w} \in \WG(d,n)$ such that $Pl_w([A]_{\sim_w})=Pl_w([B]_{\sim_w})$ for some  $A,B\in M_d(n,d)$.  Now
  \begin{equation}\label{eq_pl_inj}
  	Pl_w([A]_{\sim_w})=Pl_w([B]_{\sim_w})\implies \det(A_{\lambda^j})=t^{c_j}\det(B_{\lambda^j})
    \end{equation}
for some $t\in \CC^{*}$ and for all $j\in \{0,1,\dots,m\}$. Since $A\in M_d(n,d)$ there exists a Schubert symbol $\lambda^i= (\lambda_1^i,\dots,\lambda_d^i)$ such that $\det(A_{\lambda^i})\neq 0$.  Then using \eqref{eq_pl_inj},  $ \det(B_{\lambda^i})\neq 0 $. So we can assume $A_{\lambda^i}=B_{\lambda^i}=\text{I}_d$. Then $t^{c_i}=1$. Consider the matrices $D=\mbox{diag}(t^{w_1},t^{w_2}, \dots,t^{w_n})$ and $T=\mbox{diag}(t^{-w_{\lambda^i_1}},\dots,t^{-w_{\lambda^i_d}}) $.
Thus, we have $B_{\lambda^i}=(DAT)_{\lambda^i}.$

Let $a_{kl}$ and $b_{kl}$ be the $(k,l)$-th entries of the matrices $A$ and $B$ respectively for $k\notin(\lambda_1^i,\dots,\lambda_d^i)$ and $1\leq l\leq d$.
 For a fixed $l$, let $\lambda^j$ be the Schubert symbol obtained by replacing $\lambda_l^i$ by $k$ in $\lambda^i$ and then ordering the later set. Then $\det(A_{\lambda^j})=a_{kl}$ and $\det(B_{\lambda^j})=b_{kl}$.  Thus using \eqref{eq_pl_inj}, we get $$b_{kl}=t^{c_j}a_{kl}\implies b_{kl}=t^{c_j-c_i}a_{kl} \implies b_{kl}={t^{w_k-w_{\lambda_l^i}}}a_{kl}.$$
 The above condition holds for all $1\leq k\leq n$ and $1\leq l\leq d$. 
 This gives $ B=DAT $. Then we have $[A]_{\sim_w} = [B]_{\sim_w}$. Hence, $Pl_w$ is an injective map.

Observe that, if $W=(0,0, \dots ,0)$ and $a=1$ then the map $Pl_w$ is the usual Pl\"{u}cker map  
\begin{equation*}
	Pl \colon \G(d,n) \rightarrow \mathbb{C}P^{m}.
\end{equation*} 
It is well known that $Pl$ is an embedding. Moreover, we have the following commutative diagrams.
\begin{equation}\label{eq_diag1}
	\begin{tikzcd}
		\WG(d,n) \arrow{r}{Pl_w} & \WW P(c_0,c_1,\dots,c_m) \\%
		M_d(n,d) \arrow{r}{P} \arrow{u}{\pi_w} \arrow{d}{\pi}& \mathbb{C}^{m+1}-\{0\} \arrow{u}{\pi_c^{'}} \arrow{d}{\pi^{'}} \\%
		\G(d,n) \arrow{r}{Pl}& \mathbb{C}P^{m}.
	\end{tikzcd}
\end{equation}

Let $U$ be an open subset of $\WG(d,n)$. Then  $\pi_{w}^{-1}(U)$ is an open subset of $M_d(n,d)$. Since the map $\pi$ in \eqref{eq_map_pi} is an orbit map so $\pi(\pi_{w}^{-1}(U))$ is an open subset of $\G(d,n)$. Thus $Pl(\pi(\pi_{w}^{-1}(U)))=\pi' (P(\pi_{w}^{-1}(U)))$
is an open subset of $Pl(\G(d,n))$. Then $P(\pi_{w}^{-1}(U))$ is an open subset of $P(M_d(n,d))$. Therefore, $Pl_w(U)=\pi_c'(P(\pi_{w}^{-1}(U)))$ is an open subset of $Pl_w(\WG(d,n))$. Thus $Pl_w$ is an embedding.    
\end{proof}

Note that the actions of $(\CC^{*})^n$ on $\WG (d,n)$ and $\WW P(c_0, c_1, \ldots, c_m)$ implies that the weighted Pl\"{u}cker embedding $Pl_w$ in \eqref{weighted Pluecker map} is $(\CC^{*})^n$-equivariant, and $Pl_w(\WG (d,n))$ is a $(\CC^{*})^n$-invariant subset of $\WW P(c_0, c_1, \ldots, c_m)$. Thus all the maps in the diagram \eqref{eq_diag1} are $(\CC^*)^n$-equivariant.

Now we show that Definition \ref{def_weighted_gras} is equivalent to the definition of a weighted Grassmannian studied in \cite{AbMa1}. The algebraic torus $(\mathbb{C}^{*})^{n+1}$ acts on $\Lambda^d(\mathbb{C}^n)$ by
\begin{equation*}
(t_1, t_2, \dots, t_n, t) \sum_{i=0}^{m} a_{\lambda^i}e_{\lambda^i}=\sum_{i=0}^{m} t \cdot t_{\lambda^i} a_{\lambda^i} e_{\lambda^i}
\end{equation*} 
where $ t_{\lambda} =t_{\lambda_1} \cdots t_{\lambda_{d}}$ for  $\lambda=(\lambda_1, \ldots, \lambda_d)$. Consider the subgroup $WD$ of $(\mathbb{C}^{*})^{n+1}$ defined by 
\begin{equation*}
WD:=\{(t^{w_1},t^{w_2},\dots,t^{w_n},t^a)~|~t\in \mathbb{C}^{*}\}.
\end{equation*}
 Then the restricted action of $WD$ on $\Lambda^d(\mathbb{C}^n)-\{0\}$ is given by  $$(t^{w_1}, t^{w_2}, \dots, t^{w_n}, t^a)\sum_{i=0}^{m} a_{\lambda^i}e_{\lambda^i} = \sum_{i=0}^{m} t^{c_i}a_{\lambda^i}e_{\lambda^i}.$$ Observe that this action of $WD$ is same as the weighted $\mathbb{C}^{*}$-action in Remark \ref{rmk_wgt_pr_sp}. Then we have $\dfrac{\Lambda^d(\mathbb{C}^n)-\{0\}}{WD} =\WW P (c_0, \dots, c_m)$ and by the commutativity of the diagram \eqref{eq_diag1} we have
 \begin{equation*}
 	  Pl_w(\WG(d,n))=\dfrac{P(M_d(n,d))}{WD}. 
\end{equation*}
Therefore the topologies on $\WG(d,n)$ and $\dfrac{P(M_d(n,d))}{WD}$ are equivalent. Abe and Matsumura \cite{AbMa1} called the quotient $\dfrac{P(M_d(n,d))}{WD}$ a weighted Grassmannian and showed that it has an orbifold structure. We call $\WG(d,n)$ a weighted Grassmann orbifold associated to the pair $(W,a)$.

Next, we recall the Schubert cell decomposition of $\G(d, n)$ following \cite{MiSt}. For $k\leq n$, we identify  $$\mathbb{C}^k=\{(z_1,z_2, ...,z_k,0,...,0) \in \CC^n\}.$$ For the Schubert symbol $\lambda=(\lambda_1,\lambda_2,\dots,\lambda_d)$, the Schubert cell $E(\lambda)$ is defined by 
\begin{align*}
E(\lambda) :=\{X\in \G (d,n) ~|~  \text{dim}(X\cap \mathbb{C}^{\lambda_j})=j,~\text{dim}(X\cap \mathbb{C}^{\lambda_j-1})=j-1;  \forall~ j\in [d]\},
\end{align*}
where $[d]:=\{1,2,\dots,d\}$. We have the following homeomorphism from \cite[Chapter-6]{MiSt}.
\begin{equation}\label{mat_rep}
E(\lambda) \cong \Big{\{}\begin{bmatrix} 
	* & * & \dots &* \\
		\vdots & \vdots &   & \vdots \\
	* &   *  & \dots & * \\
	1 & 0 & \dots & 0 \\
	0 & * & \dots & * \\
	\vdots & \vdots &   & \vdots \\ 
	0 & * & \dots & * \\
	0 & 1 & \dots & 0 \\
	0 & 0 & \dots & * \\
	\vdots & \vdots &   & \vdots \\
	0 & 0 & \dots & * \\
	0 & 0 & \dots & 1 \\
	0 & 0 & \dots & 0 \\
	\vdots & \vdots &   & \vdots \\
		0 & 0 & \dots & 0 \\	
\end{bmatrix} \big{|}~ *\in \CC \text{ and }  e_j \text{ is the } \lambda_j\text{-th row for } j\in [d]\Big{\}}.
\end{equation}
 Note that $j$-th column in the matrices in \eqref{mat_rep} has $\lambda_j$-th entry $1$ and all subsequent entries of this column are zero for $j\in [d]$. Then $E(\lambda)$ is an open cell of dimension $\ell(\lambda)=(\lambda_1-1)+(\lambda_2-2)+\dots +(\lambda_d-d)$.

\begin{proposition}\label{lem_q_cell_grasm}
There is a $q$-cell structure on ${\rm WGr}(d,n)$ for $0 < d <n$.
\end{proposition}
\begin{proof}
For each $i\in \{0,1,\dots,m\}$, we define $\widetilde{E}(\lambda^i):=\pi^{-1} (E(\lambda^i))$ where the map $\pi$ is defined in \eqref{eq_map_pi}. The Schubert cell decomposition of $\G(d,n)$ gives that $\G(d,n) =\sqcup_{i=0}^m E(\lambda^i)$. This implies
\begin{equation}\label{cell structure in matrix}
 M_d(n,d)=\sqcup_{i=0}^{m}\widetilde{E}(\lambda^i),
\end{equation}
 since the map $\pi$ is surjective. Note that $$\widetilde{E}(\lambda^i)=\{A\in M_d(n,d)~|~\det(A_{\lambda^i}) \neq 0,\det(A_{\lambda^j})= 0,\text{for } j>i\}.$$ 
 Let $A\in \widetilde{E}(\lambda^i)$ and $A\sim_w B$ for a matrix $B\in M_d(n,d)$. Then $B\in \widetilde{E}(\lambda^i)$.  
Therefore we have the following decomposition of $\WG(d,n)$.
$$\WG (d,n)= \pi_w(\widetilde{E}(\lambda^0)) \sqcup \pi_w(\widetilde{E}(\lambda^1))\sqcup \dots \sqcup \pi_w(\widetilde{E}(\lambda^i)).$$
  
  By the commutativity of the diagram \eqref{eq_diag1}, we get $$Pl_w(\pi_w(\widetilde{E}(\lambda^i)))= \pi^{\prime}_c(P(\widetilde{E}({\lambda^i}))) \text { and } P(\widetilde{E}(\lambda^i)) = (\pi^{'})^{-1}(Pl(E(\lambda^i))).$$  
The map $\pi'$ is a principal $\CC^*$-bundle, and $E(\lambda^i)$ is contractible. So, there is a bundle isomorphism
  $$\phi_{i} \colon P(\widetilde{E}(\lambda^i))\to E(\lambda^i) \times \CC^{*}.$$ This map can be defined by 
  $\phi_{i}(P(A))=(\pi(A),\det(A_{\lambda^i}))$. The inverse map is defined by $(\pi(A),s)\mapsto (s(\det(A_{\lambda^i}))^{-1}P(A)) $.
  
  Let $\pi(A)\in \G(d,n)$ for some $A=({\bf a}_1,{\bf a}_2,\dots ,{\bf a}_n)^{tr}\in M_d(n,d)$ and $t\in \CC^{*}$. There is an action of $\CC^{*}$ on $\G(d,n)$ defined by
  \begin{equation}\label{weighted_T_act_Gsm}
  	t.\pi(A)=t.\pi(({\bf a}_1,{\bf a}_2,\dots ,{\bf a}_n)^{tr}):=\pi((t^{w_1}{\bf a}_1,t^{w_2}{\bf a}_2, \dots ,t^{w_n}{\bf a}_n)^{tr}).
  \end{equation}
If $\pi(A)=\pi(B)$, then $A=BT \iff DA=DBT$ for a diagonal matrix $D$ and $T\in \text{GL}(d,\CC)$. Thus $t.\pi(A)=t.\pi(B)$. 
Then $\phi_{i}$ becomes $\CC^{*}$-equivariant with the following weighted $\CC^{*}$-action on $E(\lambda^i) \times \CC^{*}$ given by $$t.(\pi(A),s)=(t.\pi(A),t^{c_i}s),$$ where $t.\pi(A)$ is defined in \eqref{weighted_T_act_Gsm} and $c_i$ is defined in \eqref{eq_wli}. Let $G(c_i)$ be the group of $c_i$-th roots of unity defined by 
  \begin{equation*}
  G(c_i) :=\{t \in \mathbb{C}^{*} ~|~ t^{c_i} = 1\},
  \end{equation*} 
   for $i=0,1,\dots,m$.
  Then the finite group $G(c_i)$ acts on the second factor of $E(\lambda^i)\times \CC^{*}$ trivially. Thus
  $$\pi^{\prime}_c(P(\widetilde{E}({\lambda^i})))= \frac{P(\widetilde{E}(\lambda^i))}{\mbox{ weighted }\CC^* \mbox{-action }}\cong \frac{E(\lambda^i) \times \CC^{*}}{\mbox{ weighted }\CC^* \mbox{-action }} \cong \frac{E(\lambda^i)}{G(c_i)}. $$ 
Therefore we get a $q$-cell decomposition of $\WG (d,n)$ given by 
 $$Pl_w(\WG (d,n))=\frac{E(\lambda^0)}{G(c_0)}\sqcup\frac{E(\lambda^1)}{G(c_1)}\sqcup \frac{E(\lambda^2)}{G(c_2)} \sqcup \dots \sqcup \frac{E(\lambda^m)}{G(c_m)}.$$
\end{proof}

 For each $k\in \{0,1,2,\dots, m\}$, let $X_k :=\sqcup_{i=0}^{k}\frac{E(\lambda^i)}{G(c_i)} \subset \WG (d, n)$. Here $X_k$ is built inductively by attaching the $q$-cells $\frac{E(\lambda^0)}{G(c_0)},\ldots,\frac{E(\lambda^k)}{G(c_k)}$ so that $X_k$ remains a subset of $\WG(d,n)$. Then we have the following filtration of $q$-CW complexes which are invariant under $(\CC^{*})^{n}$-action on $\WG (d, n)$,
\begin{equation}\label{filtration}
 \{pt\}=X_0 \subset X_1 \subset X_2 \subset \dots \subset X_m=\WG(d,n).
\end{equation} 

We note that the paper \cite{AbMa1} discussed a $q$-cell structure of $\WG (d, n)$. However, our approach is different and helps to study torsions in the integral cohomology of $\WG(d,n)$.  

\begin{remark}\label{rmk_bld_seq_mat}
   For each $k\in \{0,1,2,\dots, m\}$, consider $\widetilde{X}_k \subset M_d(n,d)$ defined by $$\widetilde{X}_k:=\{A\in M_d(n,d)~|~\text{det}(A_{\lambda^j})= 0,\text{ for } j>k\}.$$ Then $\widetilde{X}_k=\sqcup_{i=0}^{k}\widetilde{E}({\lambda^i})\subset M_d(n,d)$ and we have
	$X_k=\frac{\widetilde{X}_k}{\sim_w}=\pi_w(\widetilde{X}_k)$.
\end{remark}

\section{Integral cohomology of certain weighted Grassmann orbifolds}\label{sec_int_cohom_grassmann} 
In this section, we study several $q$-cell structure on a weighted Grassmann orbifold. We show how a permutation on the weight vector affects the weighted Grassmann orbifold. We define admissible permutation $\sigma\in S_n$ for a prime $p$ and $\WG(d,n)$. Then we discuss when $H^{*}({\rm WGr} (d,n);\mathbb{Z})$ has no $p$-torsion. We introduce the concept of divisive weighted Grassmann orbifolds which incorporates the divisive weighted projective spaces of \cite{HHRW}. We show that a divisive weighted Grassmann orbifold has a $(\CC^*)^n$-invariant cell structure. We describe this action on each cell explicitly. As a consequence, we get that the integral cohomology of a divisive weighted Grassmann orbifold has no torsion and is concentrated in even degrees.

The following lemma is well-known, but for our purpose we may need its proof. 
\begin{lemma}
The map $\pi_c' \colon \CC^{m+1} - \{0\} \to \WW P(c_0,c_1,\dots,c_m)$ induces an equivariant homeomorphism  $ \WW  P( rc_0, rc_1, \dots, rc_m) \to \WW P(c_0, c_1, \dots, c_m) $ for any positive integer $r$.
\end{lemma}

\begin{proof}
The weighted $\CC^*$-action on $\CC^{m+1}-\{0\}$ for $\WW P(rc_0, rc_1, \dots, rc_m)$ is given by 
$$t(z_0, z_1, \dots, z_m)=(t^{rc_0}z_0, t^{rc_1}z_1, \dots, t^{rc_m}z_m).$$ 
We denote the equivalence class by $[z_0: z_1: \dots: z_m]_{\sim_{rc}}$.

One can define a map $f\colon \WW P(rc_0,rc_1,\dots,rc_m) \to \WW P(c_0,\dots,c_m)$ by $$f([z_0:z_1:\dots:z_m]_{\sim_{rc}})=[z_0:z_1:\dots:z_m]_{\sim_{c}}$$ and a map $g \colon \WW P(c_0,c_1,\dots,c_m) \to \WW P(rc_0,rc_1,\dots,rc_m)$ by $$g([z_0:z_1:\dots:z_m]_{\sim_{c}})=[z_0:z_1:\dots:z_m]_{\sim_{rc}}.$$
Thus the following diagram commutes
\[\xymatrix{
	\mathbb{C}^{m+1}-\{0\} \ar[r]^{\rm{Id}} \ar[d]_{\pi'_{rc}} & \mathbb{C}^{m+1}-\{0\}  \ar[d]^{\pi'_{c}}\\
	\WW P(rc_0,\dots,rc_m) \ar@<-.5ex>[r]_f  & \WW P(c_0,\dots,c_m). \ar@<-.5ex>[l]_g }
\]

Observe that, we have $f\circ g=\text{Id}_{\WW P(c_0,\dots,c_m)}$ and $g\circ f=\text{Id}_{\WW P(rc_0,\dots, rc_m)}$. Thus $f$ is a bijective map with the inverse map $g$.

Let $U$ be an open subset of $\WW P(c_0,\dots,c_m)$ Then $(\pi_c^{'})^{-1}(U)=(\pi_{rc}^{'})^{-1}\circ f^{-1}(U)$. Since $\pi_{c}^{'}$ is a quotient map then $(\pi_{c}^{'})^{-1}(U)$ is an open subset of $\mathbb{C}^{m+1}-\{0\} $. Thus $f^{-1}(U)$ is an open subset of  $\WW P(rc_0,\dots,  rc_m)$ as $\pi_{rc}^{'}$ is a quotient map. Thus $f$ is continuous. By the similar arguments, we can show that $g$ is continuous. Hence $f$ is a homeomorphism and also it is equivariant with respect to the $(\CC^*)^n$-action on $\WW P(c_0,\dots,c_m)$ and $\WW P(rc_0,\dots,rc_m)$ defined after \eqref{T_act}.
\end{proof}

\begin{lemma}\label{lem_subset_condition}
	Let $B$ be a subset of $\mathbb{C}^{m+1}-\{0\}$. Let $B^{'}_{c}:=\pi_c^{'}(B)$ and $B^{'}_{rc}:=\pi_{rc}^{'}(B)$. Then $f|_{B^{'}_{rc}}\colon B^{'}_{rc} \to B_{c}^{'}$ is a homeomorphism.  
\end{lemma}

\begin{proof}
  Consider the following commutative diagram
   \[ \begin{tikzcd}
  	B \arrow{r}{\rm{Id}} \arrow{d}{\pi_{rc}^{'}} & B \arrow{d}{\pi_{c}^{'}} \\%
  	B^{'}_{rc} \arrow{r}{f|_{B^{'}_{rc}}}& B^{'}_{c}.
  \end{tikzcd}
  \]
The map $f$ is well defined and one-one. It follows that $f|_{B^{'}_{rc}}$ is also well defined and one-one. Note that $f|_{B^{'}_{rc}} $ is defined by $f|_{B^{'}_{rc}}(\pi_{rc}^{'}(b))=\pi_{c}^{'}(b)$. Therefore, $\pi_{rc}^{'}(b)\in B^{'}_{rc} $ is the inverse image of an element $\pi_{c}^{'}(b) \in B_{c}^{'}$. So $f|_{B^{'}_{rc}}$ is bijective. Also $(f|_{B^{'}_{rc}})^{-1} =g|_{B_{c}^{'}}.$ To conclude, $f|_{B^{'}_c}$ is a homeomorphism, recall that the restriction of a continuous map is also continuous.
\end{proof}

We apply the previous result onto some subsets of $ P(M_d(n,d))\subseteq \CC^{m+1}-\{0\}$ for $m+1={n\choose d}$. For all $k\in \{0,1,\dots,m\}$, recall the space $\widetilde{X}_k$ from Remark \ref{rmk_bld_seq_mat}. Then $$P(\widetilde{X}_k)= \sqcup_{i=0}^{k} P(\widetilde{E}(\lambda^i)) \subseteq P(M_d(n,d))$$ using \eqref{cell structure in matrix}. Also $P(\widetilde{X}_k) \subseteq \CC^{k+1}-\{0\} \subseteq \CC^{m+1}-\{0\}$, for $k\in\{0,1,\ldots,m\}$. 

One can calculate $c_i$ for all $i\in \{0,1,\dots,m\}$ from \eqref{eq_wli} for an weighted Grassmann orbifold $\WG(d,n)$. Let $r_k:=\gcd\{c_0,c_1, \dots, c_k \}$ for all $k \in \{1,2,\dots m\}$ and $G(r_k)$ be the group of $r_k$-th roots of unity. Then $G(r_k)$ is a subgroup of $G(c_i)$ and $G(c_i)/G(r_k) $ is isomorphic to $G(c_i/r_k)$ for $i \in \{0,1,2,\dots,k\}$.

\begin{lemma}\label{thm_red_loc_gp}
The space $\pi_{c}^{'}(P(\widetilde{X}_k))$ is homeomorphic to $\pi_\frac{c}{r_k}^{'}(P(\widetilde{X}_k))$. Moreover, $E(\lambda^k)/G(c_k)$ is homeomorphic to $E(\lambda^k)/G(c_k/r_k)$.   
\end{lemma}

\begin{proof}
	The following diagram is commutative. 
	\[ \begin{tikzcd}
		P(\widetilde{X}_k) \arrow{r}{\rm{Id}} \arrow{d}{\pi_{c}^{'}} & P(\widetilde{X}_k) \arrow{d}{\pi_\frac{c}{r_k}^{'}} \\%
	\pi_{c}^{'}(P(\widetilde{X}_k))	\arrow{r}{f|_{\pi_c^{'}(P(\widetilde{X}_k))}}& \pi_\frac{c}{r_k}^{'}(P(\widetilde{X}_k)).
	\end{tikzcd}
	\]
	By Lemma \ref{lem_subset_condition}, the lower horizontal map is a homeomorphism. The second statement of the Lemma follows by the similar arguments with $P(\widetilde{X}_k)$ is replaced by $P(\widetilde{E}(\lambda^k))$.
\end{proof}

\begin{theorem}\label{prop_cell_after_red}
The collection $\{\frac{E(\lambda^i)}{G(c_i/r_k)}\}_{i=0}^k$ gives a $q$-cell structure of $\pi^{'}_\frac{c}{r_k}(P(\widetilde{X}_k))$ for $k=1,2,\dots,m$. Moreover, $\{E(\lambda^i)/G(c_i/r_i)\}_{i=0}^m$ gives a $q$-cell structure of ${\rm WGr} (d, n)$ where $r_0=c_0$.
\end{theorem}	

\begin{proof}
Note that  the sets $P(\widetilde{E}(\lambda^i))$ and $P(M_d(n,d))=\sqcup_{i=0}^{m} P(\widetilde{E}(\lambda^i))$ are invariant under the weighted $\CC^*$-action defined in Remark \ref{rmk_wgt_pr_sp}. Then we have the following commutative diagram. 	
\[	
\begin{tikzcd}
	P(\widetilde{X}_k) \arrow{d}{\pi_\frac{c}{r_k}^{'}} &\subset  & \mathbb{C}^{k+1}-\{0\} \arrow{d}{\pi_\frac{c}{r_k}^{'}} \\%
	\pi_\frac{c}{r_k}^{'}(P(\widetilde{X}_k)) &\subset & \WW P(\frac{c_0}{r_k},\frac{c_1}{r_k},\dots,\frac{c_k}{r_k}).
\end{tikzcd}
\]
Thus the first part follows from
\begin{align*}
 \pi^{'}_\frac{c}{r_k}(P(\widetilde{X}_k))
&=\pi^{'}_\frac{c}{r_k}(\sqcup_{i=0}^{k}P(\widetilde{E}(\lambda^i)))\\ 
&=\sqcup_{i=0}^{k}\pi^{'}_\frac{c}{r_k}(P(\widetilde{E}(\lambda^i)))\\
&=\sqcup_{i=0}^{k}\frac{P(\widetilde{E}(\lambda^i))}{\sim_{c/r_k}} \cong \sqcup_{i=0}^{k}\frac{E(\lambda^i)}{G(c_i/r_k)}.
\end{align*}

The second part follows from $\WG(d,n)\cong \pi_c'(P(\widetilde{X}_m))$ and 
by applying Lemma \ref{thm_red_loc_gp} successively for every $k\in \{1,2,\dots,m\}$. 
\end{proof}

We show that two weighted Grassmann orbifolds are weakly equivariantly  homeomorphic if the associated weight vectors are differed by a permutation $\sigma\in S_n$. Let $X, Y$ be two $G$-spaces. A map $f \colon X \to Y$ is called a weakly equivariant  homeomorphism if $f$ is a homeomorphism and $f(g x) = \eta(g) f(x)$ for some $\eta \in {\rm Aut}(G)$ and for all $(g, x) \in G \times X$. If $\eta$ is identity, then   $f$ is called an equivariant  homeomorphism.

\begin{theorem}\label{lem_euiv_sn}
Let $W=(w_1,w_2,\dots,w_n)\in (\mathbb{Z}_{\geq 0})^n$, $0<a\in \ZZ$ and $\sigma W:=(w_{\sigma_1},w_{\sigma_2},\dots,w_{\sigma_n})$ for some $\sigma\in S_n$. If  ${\rm WGr}(d,n)$ and ${\rm WGr}^{\prime}(d,n)$ are the weighted Grassmann orbifolds associated to $(W, a)$ and $(\sigma W, a)$ respectively, then ${\rm WGr}(d,n)$ is weakly equivariantly homeomorphic to ${\rm WGr}^{\prime}(d,n)$. Moreover, this may induce different $q$-cell structures on ${\rm WGr}(d,n)$ for different $\sigma$.  	
\end{theorem}

\begin{proof}
The matrix $A =(a_{ij}) \in M_d(n,d)$ if and only if $\sigma A = (a_{\sigma_i j}) \in M_d(n,d)$. Thus the natural weakly equivariant homeomorphism $\bar{f}_{\sigma} \colon M_d(n,d) \to M_d(n,d)$ defined by $\bar{f}_{\sigma}(A)=\sigma A$ induces the following commutative diagram.
\begin{equation}\label{eq_per_diag}  
 \begin{tikzcd}
		M_d(n,d) \arrow{r}{\bar{f}_\sigma} \arrow{d}{\pi_w} &  M_d(n,d) \arrow{d}{\pi_{\sigma w}} \\%
		\WG (d,n) \arrow{r}{f_{\sigma}} & {\rm WGr}^{\prime}(d,n).
	\end{tikzcd}
\end{equation}
Here $\pi_w$ is the quotient map defined in Definition \ref{def_weighted_gras}. Thus, \eqref{eq_per_diag} induces a weakly equivariant homeomorphism $f_{\sigma} \colon \WG (d,n) \to {\rm WGr}^{\prime}(d,n)$, where $(\CC^*)^n$-action is differed by the permutation $\sigma$. Note that ${f}_{\sigma}([A]_{\sim w})=[\sigma A]_{\sim w}$.

We discuss the effects of the permutation $\sigma$ on the $q$-cell structure on $\WG(d,n)$. Consider $\mathbb{C}^i=\{(x_1,x_2,\dots,x_n) \in \CC^n ~|~ x_j=0 ~\text{for}~ j>i\}$. For $\sigma \in S_n$, define $$\sigma\mathbb{C}^n :=\{(x_{\sigma_1},x_{\sigma_2},\dots,x_{\sigma_n})\}$$ and  $$\sigma\mathbb{C}^i :=\{(x_{\sigma_1}, x_{\sigma_2}, \dots, x_{\sigma_n}) \in \sigma\mathbb{C}^n ~|~ x_{\sigma_j}=0 ~\text{for}~ \sigma_j > i\}.$$ 
Let $\lambda=(\lambda_1, \ldots, \lambda_d)$ be a Schubert symbol for $d < n$. Then 
\begin{align*}
\sigma E(\lambda)& = \{\sigma Y ~|~ Y \in E(\lambda)\}\\
& =\{X\in \G(d,n)~|~\text{dim}(X\cap \sigma\mathbb{C}^{\lambda_i})=i,~ \text{dim}(X\cap \sigma\mathbb{C}^{\lambda_i-1})=i-1, i \in [d]\}
\end{align*}
where $[d] = \{1,2,\ldots,d\}$. 
 Then $E(\lambda) \cong \sigma E(\lambda)$ and dim$(\sigma E(\lambda))=\ell(\lambda)$.
 
So the permutation of the coordinates in $\CC^n$ determines another cell structure for $\G(d,n)$ given by $\G(d,n)=\sigma\G(d,n)=\sqcup_{i=0}^{m} \sigma E(\lambda^i)$. This cell structure of $\G(d,n)$ induces the following decomposition of $M_d(n,d)$ which is similar to \eqref{cell structure in matrix}. 
\begin{equation*}
	M_d(n,d)=\sqcup_{i=0}^{m}\sigma\widetilde{E}(\lambda^i) ~\mbox{and} ~ P(M_d(n,d))=\sqcup_{i=0}^{m}P(\sigma\widetilde{E}(\lambda^i)).
\end{equation*}
 
Recall that $\lambda^i =(\lambda_1^i, \ldots, \lambda_d^i)$ is a Schubert symbol and $c_i$ is defined in \eqref{eq_wli} for $i=0, \ldots, m$. Then $\sigma\lambda^i :=(\sigma(\lambda_{i_1}^i), \ldots, \sigma(\lambda_{i_d}^i))$, where $i_1,\dots,i_d\in\{1,\dots,d\}$ such that $\sigma(\lambda_{i_1}^i)<\sigma(\lambda_{i_2}^i)<\cdots <\sigma(\lambda_{i_d}^i)$. Let 
\begin{equation}\label{eq_sigma_c_j}
	\sigma c_i:=a+\sum_{j=1}^dw_{\sigma(\lambda_{i_j}^i)}.
\end{equation}
 Now from the commutativity of the diagram \eqref{eq_diag1},  we have the following.
$$\pi_w(\sigma(\widetilde{E}(\lambda^i))) \cong Pl_w(\pi_w(\sigma\widetilde{E}(\lambda^i)))= \frac{P(\sigma\widetilde{E}(\lambda^i))}{\mbox{ weighted }\CC^* \mbox{-action }} \cong \frac{\sigma E(\lambda^i)}{G(\sigma c_i)}. $$ 
Then we get a $q$-cell structure of the weighted Grassmann orbifold $\WG(d,n)$ given by $$\WG(d,n)\cong \frac{\sigma E(\lambda^0)}{G(\sigma c_0)}\sqcup \frac{\sigma E(\lambda^1)}{G(\sigma c_1)} \sqcup \dots \sqcup \frac{\sigma E(\lambda^m)}{G(\sigma c_m)}.$$
\end{proof}

\begin{remark}
	 Applying the permutation $\sigma$ on the rows of the matrices in $E(\lambda)$, we get the matrices of $\sigma E(\lambda)$. That is,
	$$\begin{pmatrix}
		v_1  \\
		v_2  \\
		\vdots    \\
		v_n  
	\end{pmatrix} \in E(\lambda) \iff  \begin{pmatrix}
		v_{\sigma_1}  \\
		v_{\sigma_2}  \\
		\vdots    \\
		v_{\sigma_n}   
	\end{pmatrix} \in \sigma E(\lambda).$$ 
\end{remark}
	
\begin{proposition}\cite[Theorem 1.1]{BNSS}\label{thm_bnss}
	Let $X$ be a $q$-CW complex with no odd dimensional $q$-cells and $p$
	a prime number. Let $\{pt\}=X_0\subseteq X_1\subseteq \dots \subseteq X_s=X$ is a filtration of $X$ such that $X_i$ is obtained by attaching the $q$-cell $\RR^{2k_i} /G_i$ to $X_{i-1}$ for all $i\in \{1,2,\dots,s\}$. If $\gcd\{p,|G_i|\}= 1$ for all $i\in \{1,2,\dots,s\}$,  then $H^{*}(X; \ZZ)$ has no $p$-torsion and $H^{odd}(X; \ZZ_p)$ is trivial.
\end{proposition}

Recall the definition of $\sigma c_i$ from \eqref{eq_sigma_c_j} for $\WG(d,n)$ associated to the weight $W=(w_1,\dots,w_n)\in (\ZZ_{\geq0})^n$ and $1\leq a\in \ZZ$.
\begin{definition}\label{def_adm_per}
 A permutation $\sigma \in S_n$ is called admissible for a prime $p$ and ${\rm WGr}(d,n)$ if 
 $$\gcd\{p,\frac{\sigma c_i}{d_i}\}=1$$
 where  $\sigma c_i$ is defined in \eqref{eq_sigma_c_j} and $d_i=\gcd\{\sigma c_0,\sigma c_1,\dots,\sigma c_i\}$, for $i\in \{1,2,\dots, m\}$.
\end{definition}
Some examples of admissible permutations are discussed in Example \ref{ex:no_torsion2}.

\begin{remark}\label{rmk_adm_per}
 There may not always exist an admissible permutation $\sigma \in S_n$ for  a prime $p$ and ${\rm WGr}(d,n)$. However if $d=1$, then $m=n-1$ and there always exists an admissible permutation $\sigma \in S_n$ for every prime $p$. The admissible permutation $\sigma \in S_n$ may not be unique.
\end{remark}

Now we prove the following result which says when the integral cohomology of $\WG(d,n)$ have no $p$-torsion. 

\begin{theorem}\label{p torsion in permuted Grassmann}
 If there exists an admissible permutation $\sigma \in S_n$ for a prime $p$ and ${\rm WGr} (d,n)$, then $H^{*}({\rm WGr} (d,n);\mathbb{Z})$ has no $p$-torsion and $H^{odd}({\rm WGr} (d,n);\mathbb{Z}_p)$ is trivial. 
 \end{theorem}

\begin{proof} 
Suppose $\sigma \in S_n$ be the admissible permutation for $p$ and $\WG(d,n)$. Then $$\gcd\{p,\frac{\sigma c_i}{d_i}\}=1 $$
 by Definition \ref{def_adm_per}, where $d_i=\gcd\{\sigma c_0, \sigma c_1, \dots, \sigma c_i\}$ for all $i\in \{1,2,\dots,m\}$. By Theorem \ref{lem_euiv_sn}, we have the following $q$-cell structure  $$\WG (d,n)\cong \frac{\sigma E(\lambda^0)}{G(\sigma c_0)}\sqcup \frac{\sigma E(\lambda^1)}{G(\sigma c_1)} \sqcup \dots \sqcup \frac{\sigma E(\lambda^m)}{G(\sigma c_m)},$$
where $\sigma E(\lambda^i)\cong E(\lambda^i)\cong \CC^{\ell(\lambda^i)}$. Let 
$$\sigma X_k = \sqcup_{i=0}^k \frac{\sigma E(\lambda^i)}{G(\sigma c_i)} \subseteq \WG(d,n) \text{ for } k =0, 1,\ldots, m.$$ 
Then $\sigma X_k$ is a subcomplex of $\WG(d,n)$ for $k=0,1,\dots,m$ and $\sigma X_m=\WG(d,n)$. This gives a filtration 

$$\{pt\}=\sigma X_0\subset \sigma X_1\subset \dots\subset \sigma X_m=\WG(d,n)$$
such that $\sigma X_i-\sigma X_{i-1}$ is homeomorphic to $\frac{\sigma E(\lambda^{i})}{G(\sigma c_i)}$.

 Using Lemma \ref{thm_red_loc_gp}, $\frac{\sigma E(\lambda^{i})}{G(\sigma c_i)}\cong \frac{\sigma E(\lambda^{i})}{G(\frac{\sigma c_i}{d_i})}$. That is $\sigma X_i-\sigma X_{i-1}$ is homeomorphic to $\frac{\CC^{\ell(\lambda^i)}}{G(\frac{\sigma c_i}{d_i})}$ for all $i=1,2,\dots,m$. Therefore, by Proposition \ref {thm_bnss}, $H^{*}(\WG(d,n);\ZZ)$ has no $p$-torsion and the group $H^{odd}(\WG (d,n);\mathbb{Z}_p)$ is trivial. This completes the proof.
\end{proof}

\begin{corollary}\cite{Ka}
$H^{*}(\WW P(c_0,c_1,\dots,c_m);\mathbb{Z})$ has no torsion.
\end{corollary}

\begin{proof}
	This follows from Theorem \ref{p torsion in permuted Grassmann}, Remark \ref{rmk_wgt_pr_sp} and \ref{rmk_adm_per}.
\end{proof}

\begin{example}\label{ex:no_torsion2}
Consider the weighted Grassmann orbifold $\WG(2,4)$ for $W=(1,1,3,4)$ and $a=2$.  Here  $n=4,~d=2,~{n \choose d}=6,~ m={n \choose d}-1=5$. So in this case, we have $6$ Schubert symbols which are 
\begin{equation*}
	\lambda^0=(1,2)<\lambda^1=(1,3)<\lambda^2=(1,4)<\lambda^3=(2,3)<\lambda^4=(2,4)<\lambda^5=(3,4)
\end{equation*}
	in the ordering as in Definition \ref{def:dict_order}. For the prime $p=3$, consider the permutation $\sigma\in S_4$ defined by $$\sigma_1=3, \sigma_2=4, \sigma_3=1, \sigma_4=2 .$$ Then $$\sigma c_0=9, \sigma c_1=6, \sigma c_2=6, \sigma c_3=7, \sigma c_4=7 \text{ and } \sigma c_5=4$$ using \eqref{eq_sigma_c_j}. This $\sigma$ is admissible for $p=3$ and $\WG(2,4)$. Thus $H^{*}(\WG(2,4);\mathbb{Z})$ has no $3$-torsion by Theorem \ref{p torsion in permuted Grassmann}.
	
For the prime $p=7$, consider  the permutation  $\sigma\in S_4$ defined by $$\sigma_1=4, \sigma_2=2, \sigma_3=1, \sigma_4=3.$$ 
Then
 $$ \sigma c_0=7, \sigma c_1=7, \sigma c_2=9, \sigma c_3=4, \sigma c_4=6 \text{ and } \sigma c_5=6$$ using \eqref{eq_sigma_c_j}. This $\sigma$ is admissible for $p=7$ and $\WG(2,4)$. Thus $H^{*}(\WG(2,4);\mathbb{Z})$ has no $7$-torsion by Theorem \ref{p torsion in permuted Grassmann}.

 To compute that it has no 2-torsion, we need to consider a different total order on the Schubert symbols given by 
 $$\lambda^0=(1,2)<\lambda^1=(1,3)<\lambda^2=(2,3)<\lambda^3=(1,4)<\lambda^4=(2,4)<\lambda^5=(3,4) $$
 which preserves the partial order in \eqref{eq_bru_ord}. In this case, $$c_0=4,c_1=6,c_2=6,c_3=7,c_4=7 \text{ and } c_5=9$$ using \eqref{eq_wli}. 
  The identity permutation in $S_4$ is admissible for $p=2$ and this $\WG(2,4)$. Then $H^{*}(\WG(2,4);\mathbb{Z})$ has no $2$-torsion by Theorem \ref{p torsion in permuted Grassmann}.
	
The only primes which divides the orders of the orbifold singularities of this $\WG(2,4)$ are $2,3$ and $7$. Hence the integral cohomology of $\WG(2,4)$ of this example has no torsion. \qed
\end{example}

\begin{remark}
	Considering the total order given in Definition \ref{def:dict_order} on the Schubert symbols, there may not exist an admissible permutation $\sigma$ for a prime. However, one can take another total order on the Schubert symbols for which one can find $\sigma$  satisfying the hypothesis in Theorem \ref{p torsion in permuted Grassmann} for this prime.
\end{remark}

The $q$-cell structure in Theorem \ref{prop_cell_after_red} leads us to introduce the following definition which generalizes the concept of divisive weighted projective spaces of \cite{HHRW}. 

\begin{definition}\label{def:div_wgs}
	A weighted Grassmann orbifold ${\rm WGr}(d,n)$ is called divisive if there exists $\sigma \in S_n$ such that ${\sigma c_{i}}$ divides ${\sigma c_{i-1}}$ for $i=1,2,\dots,m$ where $\sigma c_i$ is defined in \eqref{eq_sigma_c_j}. 
\end{definition}

\begin{example}
	Consider the weighted Grassmann orbifold $\WG (2,4)$ for the weight $W=(1,6,1,1)$ and $a=3$. We have the ordering on the $6$ Schubert symbols given by $$\lambda^0=(1,2)<\lambda^1=(1,3)<\lambda^2=(1,4)<\lambda^3=(2,3)<\lambda^4=(2,4)<\lambda^5=(3,4).$$ Consider the permutation $\sigma\in S_4$ defined by $$\sigma_1=2,\sigma_2=1,\sigma_3=3,\sigma_4=4.$$ Then $$\sigma c_{0}=10, \sigma c_{1}=10, \sigma c_{2}=10, \sigma c_{3}=5, \sigma c_{4}=5, \sigma c_{5}=5$$  using \eqref{eq_sigma_c_j}. Thus $\sigma c_{i}$ divides $\sigma c_{i-1}$ for $i=1,2,\dots,5$. So $\WG(2,4)$ of this example is divisive. \qed
\end{example}

\begin{example}
	Let  $\alpha$ and $\gamma$ be any two non-negative integers and $\beta$ be any positive integer such that $\beta>d\alpha$. Let $\WG(d,n)$ be the corresponding weighted Grassmann orbifold for $W=(\alpha+\gamma\beta,\alpha,\dots,\alpha)\in (\ZZ_{\geq0})^n$  
	and $a=\beta-d\alpha >0$. Consider the total order $\{\lambda^0,\lambda^1,\dots,\lambda^m\}$ on the Schubert symbol induced by the dictionary order. Then 
\begin{align*}
c_i= 
     \begin{cases}
       (\gamma+1)\beta  & \quad\text{if}~  i=0,1,\dots, {n-1 \choose d-1}-1 \\
           \beta  &\quad\text{if}~ i={n-1 \choose d-1},\dots,m.
     \end{cases}
\end{align*}
Then $c_i$ divides $c_{i-1}$ for all $i=1,2,\dots,m$. Therefore this $\WG (d,n)$ is a divisive weighted Grassmann orbifold. \qed
\end{example}

\begin{definition}
 Let $\lambda$ be a Schubert symbol  for $d < n$. Then a reversal of $\lambda$ is a pair  $(k, k')$ such that $k \in \lambda$, $k' \notin \lambda$ and $k' < k $. We denote the set of all reversals of $\lambda$ by $\rm{rev}(\lambda)$. If $(k,k')\in \rm{rev}(\lambda)$ then $(k,k')\lambda$ is the Schubert symbol obtained by replacing $k$ by $k'$ in $\lambda$ and ordering the later set. 	
\end{definition}

\begin{remark}
If $(k, k') \in \text{rev}(\lambda)$ then $(k, k')\lambda \prec \lambda$ and $\ell(\lambda)$ is the cardinality of the set $\text{rev}(\lambda)$  where $\ell(\lambda)$ is the length of $\lambda$. In \cite{KnTa, AbMa1} the authors defined an inversion of a Schubert symbol $\lambda$ is a pair $(k, k')$ such that $k \in \lambda, k' \notin \lambda$ and $k < k'$. In some sense, our  definition of reversal is dual to the definition of inversion. If $\rm{inv}(\lambda)$ be the set of all inversions of $\lambda$ and $\ell^{'}(\lambda)$ is the cardinality of the set $\rm{inv}(\lambda)$ then $\ell(\lambda)+\ell^{'}(\lambda)=d(n-d)$. Also If $(k,k')\in \rm{rev}(\lambda)$ and $(k,k')\lambda=\mu$ then $(k',k)\in \rm{inv}(\mu)$ and and $(k',k)\mu=\lambda$.	
\end{remark}

Next we discuss $(\CC^*)^n$-action on some cell structure of a divisive weighted Grassmann orbifold. Recall the $(\CC^*)^n$-action on $\WG(d,n)$ which is induced from \eqref{T_act_Gsm}. We adhere the notation from Section \ref{sec_weighted_grasm}.

\begin{theorem}\label{thm:tinv_cell}
If ${\rm WGr}(d,n)$ is a divisive weighted Grassmann orbifold then it has a $(\CC^*)^n$-invariant cell structure with only even dimensional cells. 
\end{theorem}

\begin{proof}
Let $\WG(d, n)$ be a divisive weighted Grassmann orbifold corresponding to $W = (w_1, \ldots, w_n)\in (\ZZ_{\geq 0})^n$ and $1\leq a \in \ZZ$. Then there exists $\sigma\in S_n$ such that $\sigma c_i$ divides $\sigma c_{i-1}$ for all $i=1,2,\dots,m$. Let us assume $\sigma=\text{Id}$ (the identity permutation in $S_n$). Then $c_i$ divides $c_{i-1}$ for all $i=1,2,\dots,m$. Then $\gcd\{c_0,c_1,\dots,c_i\}=c_i$ for all $i\in \{1,2,\dots,m\}$. Thus 
 $$\pi_{w}(\widetilde{E}(\lambda^i)) \cong \frac{E(\lambda^i)}{G(c_i)} \cong \frac{E(\lambda^i)}{G(c_i/c_i)}\cong E(\lambda^i) \text{ for all } i=1,2,\dots,m$$ 
 by Lemma \ref{thm_red_loc_gp}. Thus each element of $\pi_{w}(\widetilde{E}(\lambda^i))$ can be represented uniquely by the equivalence class of an $n \times d$ matrix defined in \eqref{mat_rep}.

Let $\lambda^i=(\lambda_1, \ldots, \lambda_d)$ be a Schubert symbol for $d <n$ and ${\bf z} \in \mathbb{C}^{\ell(\lambda^i)}$. Since $\ell(\lambda^i)=(\lambda_1-1)+(\lambda_2-2)+\dots +(\lambda_d-d)$, we can write
 $${\bf z}=({\bf z}_1, {\bf z}_2, \ldots, {\bf z}_d)$$
where ${\bf z}_l= (z^l_1, z^l_2, \dots, \widehat{z^l_{\lambda_1}}, \dots, \widehat{z^l_{\lambda_2}},\dots,\widehat{z^l_{\lambda_{l-1}}},\dots, z^l_{\lambda_l-1})$ for $l=1, \ldots, d$. 

For $(t_1, \ldots, t_n) \in (\CC^*)^{n}$, we define $s \in \CC^*$ such that $s^{c_i}= t_{\lambda_1} \cdots t_{\lambda_d} $. Define $T \in \mbox{GL}(d,\mathbb{C})$ by $$T=\mbox{diag}((\dfrac{t_{\lambda_1}}{s^{w_{\lambda_1}}}), (\dfrac{t_{\lambda_2}}{s^{w_{\lambda_2}}}), \dots,(\dfrac{t_{\lambda_d}}{s^{w_{\lambda_d}}})).$$ Then $\det(T)=s^a$. 

Define $g_{\lambda^i} \colon \mathbb{C}^{\ell(\lambda^i)} \to \pi_w(\widetilde{E}(\lambda^i))$ by 
 
\[g_{\lambda^i}({\bf z}):=
\begin{bmatrix} 
	z^1_1 & z^2_1 & \dots & z^d_1 \\
	\vdots & \vdots &   & \vdots \\ 
	z^1_{\lambda_1-1} &   z^2_{\lambda_1-1}  & \dots & z^d_{\lambda_1-1} \\
	1 & 0 & \dots & 0 \\
	0 & z^2_{\lambda_1+1} & \dots & z^d_{\lambda_1+1} \\
	\vdots & \vdots &   & \vdots \\ 
	0 & z^2_{\lambda_2-1}& \dots & z^d_{\lambda_2-1} \\
	0 & 1 & \dots & 0 \\
	0 & 0 & \dots & z^d_{\lambda_2+1} \\
	\vdots & \vdots &   & \vdots \\
	0 & 0 & \dots & z^d_{\lambda_d-1} \\
	0 & 0 & \dots & 1 \\
	0 & 0 & \dots & 0 \\
	\vdots & \vdots &   & \vdots \\
	0 & 0 & \dots & 0 \\
	
\end{bmatrix}.
\]

 Then $g_{\lambda^i}$ is a homeomorphism. Now we have
\[
(t_1,t_2,\dots,t_n)g_{\lambda^i}({\bf z}) =
\begin{bmatrix} 
	t_1z^1_1 & t_1z^2_1 & \dots & t_1z^d_1 \\
	\vdots & \vdots &   & \vdots \\ 
	t_{\lambda_1-1}z^1_{\lambda_1-1} & t_{\lambda_1-1} z^2_{\lambda_1-1}  & \dots & t_{\lambda_1-1}z^d_{\lambda_1-1} \\
	t_{\lambda_1} & 0 & \dots & 0 \\
	0 & t_{\lambda_1+1}z^2_{\lambda_1+1} & \dots & t_{\lambda_1+1}z^d_{\lambda_1+1} \\
	\vdots & \vdots &   & \vdots \\ 
	0 & t_{\lambda_2-1}z^2_{\lambda_2-1}& \dots & t_{\lambda_2-1}z^d_{\lambda_2-1} \\
	0 & t_{\lambda_2} & \dots & 0 \\
	0 & 0 & \dots & t_{\lambda_2+1}z^d_{\lambda_2+1} \\
	\vdots & \vdots &   & \vdots \\
	0 & 0 & \dots & t_{\lambda_d-1}z^d_{\lambda_d-1} \\
	0 & 0 & \dots & t_{\lambda_d} \\
	0 & 0 & \dots & 0 \\
	\vdots & \vdots &   & \vdots \\
	0 & 0 & \dots & 0 \\
\end{bmatrix}.
\]

Then

\[
 (t_1,t_2,\dots,t_n)g_{\lambda^i}({\bf z})=
\begin{bmatrix} 
	\frac{s^{w_{\lambda_1}}}{t_{\lambda_1}}t_1z^1_1 & \frac{s^{w_{\lambda_2}}}{t_{\lambda_2}}t_1z^2_1 & \dots & \frac{s^{w_{\lambda_d}}}{t_{\lambda_d}}t_1z^d_1 \\
	\vdots & \vdots &   & \vdots \\ 
	\frac{s^{w_{\lambda_1}}}{t_{\lambda_1}}t_{\lambda_1-1}z^1_{\lambda_1-1} & \frac{s^{w_{\lambda_2}}}{t_{\lambda_2}}t_{\lambda_1-1} z^2_{\lambda_1-1}  & \dots & \frac{s^{w_{\lambda_d}}}{t_{\lambda_d}}t_{\lambda_1-1}z^d_{\lambda_1-1} \\
	\frac{s^{w_{\lambda_1}}}{t_{\lambda_1}}t_{\lambda_1} & 0 & \dots & 0 \\
	0 & \frac{s^{w_{\lambda_2}}}{t_{\lambda_2}}t_{\lambda_1+1}z^2_{\lambda_1+1} & \dots & \frac{s^{w_{\lambda_d}}}{t_{\lambda_d}}t_{\lambda_1+1}z^d_{\lambda_1+1} \\
	\vdots & \vdots &   & \vdots \\ 
	0 &\frac{s^{w_{\lambda_2}}}{t_{\lambda_2}} t_{\lambda_2-1}z^2_{\lambda_2-1}& \dots & \frac{s^{w_{\lambda_d}}}{t_{\lambda_d}}t_{\lambda_2-1}z^d_{\lambda_2-1} \\
	0 & \frac{s^{w_{\lambda_2}}}{t_{\lambda_2}}t_{\lambda_2} & \dots & 0 \\
	0 & 0 & \dots & \frac{s^{w_{\lambda_d}}}{t_{\lambda_d}}t_{\lambda_2+1}z^d_{\lambda_2+1} \\
	\vdots & \vdots &   & \vdots \\
	0 & 0 & \dots & \frac{s^{w_{\lambda_d}}}{t_{\lambda_d}}t_{\lambda_d-1}z^d_{\lambda_d-1} \\
	0 & 0 & \dots & \frac{s^{w_{\lambda_d}}}{t_{\lambda_d}}t_{\lambda_d} \\
	0 & 0 & \dots & 0 \\
	\vdots & \vdots &   & \vdots \\
	0 & 0 & \dots & 0 \\
\end{bmatrix} \times T.
\]

\begin{align*}
 & \text{ Thus, } (t_1,t_2,\dots,t_n)g_{\lambda^i}({\bf z})\\ 
 & = D \times 
\begin{bmatrix} 
	\frac{s^{w_{\lambda_1}}}{t_{\lambda_1}s^{w_1}}t_1z^1_1 & \frac{s^{w_{\lambda_2}}}{t_{\lambda_2}s^{w_1}}t_1z^2_1 & \dots & \frac{s^{w_{\lambda_d}}}{t_{\lambda_d}s^{w_1}}t_1z^d_1 \\
	\vdots & \vdots &   & \vdots \\ 
	\frac{s^{w_{\lambda_1}}}{s^{w_{\lambda_1-1}}t_{\lambda_1}}t_{\lambda_1-1}z^1_{\lambda_1-1} & \frac{s^{w_{\lambda_2}}}{s^{w_{\lambda_1-1}}t_{\lambda_2}}t_{\lambda_1-1} z^2_{\lambda_1-1}  & \dots & \frac{s^{w_{\lambda_d}}}{s^{w_{\lambda_1-1}}t_{\lambda_d}}t_{\lambda_1-1}z^d_{\lambda_1-1} \\
	1 & 0 & \dots & 0 \\
	0 & \frac{s^{w_{\lambda_2}}}{s^{w_{\lambda_1+1}}t_{\lambda_2}}t_{\lambda_1+1}z^2_{\lambda_1+1} & \dots & \frac{s^{w_{\lambda_d}}}{s^{w_{\lambda_1+1}}t_{\lambda_d}}t_{\lambda_1+1}z^d_{\lambda_1+1} \\
	\vdots & \vdots &   & \vdots \\ 
	0 &\frac{s^{w_{\lambda_2}}}{s^{w_{\lambda_2-1}}t_{\lambda_2}} t_{\lambda_2-1}z^2_{\lambda_2-1}& \dots & \frac{s^{w_{\lambda_d}}}{s^{w_{\lambda_2-1}}t_{\lambda_d}}t_{\lambda_2-1}z^d_{\lambda_2-1} \\
	0 & 1  & \dots & 0 \\
	0 & 0 & \dots & \frac{s^{w_{\lambda_d}}}{s^{w_{\lambda_2+1}}t_{\lambda_d}}t_{\lambda_2+1}z^d_{\lambda_2+1} \\
	\vdots & \vdots &   & \vdots \\
	0 & 0 & \dots & \frac{s^{w_{\lambda_d}}}{s^{w_{\lambda_d-1}}t_{\lambda_d}}t_{\lambda_d-1}z^d_{\lambda_d-1} \\
	0 & 0 & \dots & 1 \\
	0 & 0 & \dots & 0 \\
	\vdots & \vdots &   & \vdots \\
	0 & 0 & \dots & 0 \\
\end{bmatrix} \times T\\
&= D M T,
\end{align*}
where $D = \mbox{diag}(s^{w_1}, \ldots, s^{w_n})$ is a diagonal matrix.
So by the equivalence relation $\sim_w$ as in Definition \ref{def_weighted_gras}, we get $$(t_1,t_2,\dots,t_n)g_{\lambda^i}({\bf z}) = M \in \pi_{w}(\widetilde{E}(\lambda^i))\subset\WG(d, n).$$

Let $a_{kl}$ be the coefficient of $z^l_k$ in the matrix $M$ for $1\leq l \leq d$, $1\leq k \leq \lambda_l-1$, $k\neq \lambda_1, \lambda_2, \dots,\lambda_{l-1}$. Then 
$$a_{kl}=\frac{s^{w_{\lambda_l}} t_k}{s^{w_k}t_{\lambda_l}}.$$
 
Now for $1\leq k \leq \lambda_l-1$, $k\neq \lambda_1, \lambda_2, \dots,\lambda_{l-1}$ we have $(\lambda_l,k)\in \rm{rev}(\lambda^i)$. Let  $\lambda^{j}=(\lambda_l,k)\lambda^i$. Note that $\lambda^j<\lambda^i$. Recall $c_i$ from \eqref{eq_wli}. 
So $$\frac{t_k s^{w_{\lambda_l}}}{s^{w_k}t_{\lambda_l}}
=\frac{t_{\lambda^{j}}} {t_{\lambda^i}}  s^{w_{\lambda_l} - w_k}
=\frac{t_{\lambda^{j}}} {t_{\lambda^i}}  s^{c_i - c_j}
=\frac{t_{\lambda^{j}}}{t_{\lambda^i}}  t_{\lambda^i}^{\frac{c_i - c_j}{c_i}}
= t_{\lambda^{j}}(t_{\lambda^i})^{-\frac{c_j}{c_i}},$$ since $s^{c_i}=t_{\lambda_1} \cdots t_{\lambda_d} = t_{\lambda^i}$
 and $t_{\lambda^{j}}=t_{\lambda_1} \cdots t_{\lambda_{l-1}} t_k t_{\lambda_{l+1}} \cdots t_{\lambda_d}$. Since $\WG (d,n)$ is divisive and $\lambda^{j} < \lambda^i$, we have $c_i$ divides $c_j$.

Define a $(\CC^{*})^n$-action on $\mathbb{C}^{\ell(\lambda^i)}$ by 
\begin{equation*}
(t_1,t_2,\dots,t_n)(z^l_k)=(t_{\lambda^{j}}(t_{\lambda^i})^{-\frac{c_j}{c_i}}z^l_k)
\end{equation*}
for $1\leq l \leq d; 1\leq k \leq \lambda_l-1; k\neq \lambda_1,\lambda_2,\dots,\lambda_{l-1}$. With this action of $(\CC^{*})^n$ on $\mathbb{C}^{\ell(\lambda^i)}$, the map $g_{\lambda^i}$ becomes $(\CC^{*})^n$-equivariant.

If $\sigma\neq \mbox{Id}$, then consider the cell $$\pi_{w}(\sigma\widetilde{E}(\lambda^i)) \cong \frac{\sigma E(\lambda^i)}{G(\sigma c_i)} \cong  \frac{\sigma E(\lambda^i)}{G(\sigma c_i/\sigma c_i)}\cong \sigma E(\lambda^i),\text{ for all } i=1,2,\dots,m $$ by Lemma \ref{thm_red_loc_gp}. Thus we get the map $\sigma g_{\lambda^i} \colon \mathbb{C}^{\ell(\lambda^i)} \to \pi_w(\sigma \widetilde{E}(\lambda^i))$ defined by ${\bf z}\to \sigma g_{\lambda^i}({\bf z})$. Then by similar arguments, we get the $(\CC^{*})^n$-action on $\mathbb{C}^{\ell(\lambda^i)}$ defined by 
\begin{equation}\label{eq:char_edge}
	(t_1,t_2,\dots,t_n)(z^l_k)=(t_{\sigma\lambda^{j}}(t_{\sigma\lambda^i})^{-\frac{\sigma c_j}{\sigma c_i}}z^l_k).
\end{equation}
\end{proof} 

\begin{corollary}\label{cor:cohom_wgs}
	If ${\rm WGr}(d,n)$ is divisive, then $H^{*}({\rm WGr}(d,n); \mathbb{Z})$ has no torsion and is concentrated in even degrees. 
\end{corollary}

We remark that Corollary \ref{cor:cohom_wgs} also follows from the proof of Theorem \ref{p torsion in permuted Grassmann} and Definition \ref{def:div_wgs}. However, Theorem \ref{thm:tinv_cell} describes the representation of the $(\CC^{*})^n$-action on each invariant cell explicitly. We also get that a divisive weighted Grassmann orbifold is integrally equivariantly formal.

\section{Equivariant cohomology, cobordism and $K$-theory of weighted Grassmann orbifolds}\label{sec_kth_cobth}

	In this section, first we compute the equivariant $K$-theory ring of any weighted Grassmann orbifold with rational coefficients. Then we compute the equivariant cohomology ring, equivariant $K$-theory ring and equivariant cobordism ring of a divisive weighted Grassmann orbifold with integer coefficients. We discuss the computation of the equivariant Euler classes for some line bundles on a point. We also compute the integral equivariant cohomology ring of some non-divisive weighted Grassmann orbifolds. We adhere the notations of previous sections. 
	
We recall the $(\CC^*)^n$-action on $\WG(d,n)$ which is induced by \eqref{T_act_Gsm}. Consider the standard torus $T^n = (S^1)^n \subset (\CC^*)^n$.	So we have the restricted $T^n$-action on $\WG (d,n)$. For each Schubert symbol $\lambda= (\lambda_1, \lambda_2,\dots, \lambda_d)$, let
$C(\lambda)\in M_d(n, d)$ with column vectors given by $e_{\lambda_1}, e_{\lambda_2} ,\dots, e_{\lambda_d}$ where $\{e_1, e_2, \dots, e_n\}$ is the standard basis for $\CC^n$. Therefore $[C(\lambda)] \in \WG(d, n)$ and it is a fixed point of the $T^n$-action on $\WG(d,n)$.  

\begin{proposition}\label{thm:q-cw_srtucture}
Let ${\rm WGr}(d,n)$ be a weighted Grassmann orbifold corresponding to $W=(w_1,w_2,\dots,w_n)\in(\mathbb{Z}_{\geq0})^n~ \mbox{and}~ a\geq 1$. Then there is a $(\CC^*)^n$-invariant stratification 
\begin{equation*}
\{pt\} = X_0 \subset X_1 \subset X_2 \subset \cdots \subset X_{m}={\rm WGr}(d,n)
\end{equation*}
such that the quotient $X_{i}/X_{i-1}$ 
is homeomorphic to the Thom space $Th(\xi^{i})$
of an orbifold $(\CC^*)^n$-vector bundle 
\begin{equation}\label{eq_orb_vector_bdl_over_pt}
\xi^{i} \colon \CC^{\ell(\lambda^i)}/G(c_i) \to [C(\lambda^i)],
\end{equation}
where $G(c_i)$ is the cyclic group of the $c_i$-th roots of unity,  for $i=1, \dots, m$. 
\end{proposition}
\begin{proof}
Recall the  $(\CC^*)^n$-invariant stratification $$\{pt\} = X_0 \subset X_1 \subset X_2 \subset \cdots \subset X_{m}=\WG(d,n)$$ from \eqref{filtration} which is obtained from the $q$-cell structure of $\WG(d,n)$ as in Lemma \ref{lem_q_cell_grasm}.   
Note that $X_i/X_{i-1}$ is the one point compactification of $\frac{E(\lambda^i)}{G(c_i)}$ which is the Thom space of the orbifold $(\CC^*)^n$-vector bundle $$\frac{E(\lambda^i)}{G(c_i)} \to [C(\lambda^i)]$$ where $[C(\lambda^i)]$ is the $(\CC^*)^n$-fixed point corresponding to the Schubert symbol $\lambda^i$ for $i=1, \ldots, m$. It remains to note that $E(\lambda^i)$ is $(\CC^*)^n$-equivariantly homeomorphic to $\CC^{\ell(\lambda^i)}$, see \eqref{mat_rep}. 
\end{proof}

Now corresponding to $\text{rev}(\lambda^i)$, one can define a subset of Schubert symbols as follows 
\begin{equation}\label{eq_rev_lbd}
	R(\lambda^i):=\{\lambda^j~|~\lambda^j=(k,k')\lambda^i \text{ for } (k,k')\in \text{rev}(\lambda^i)\}.
\end{equation}
 Then the cardinality of the set $R(\lambda^i)$ is $\ell(\lambda^i)$ for every $i\in \{0,1,\dots,m\}$. 
Note that the bundle in \eqref{eq_orb_vector_bdl_over_pt} is also an orbifold $T^n$-bundle.
\begin{proposition}\label{prop:cond_hhh}
The orbifold $T^n$-bundle in \eqref{eq_orb_vector_bdl_over_pt} has a decomposition $$\xi^{i} \colon \frac{\CC^{\ell(\lambda^i)}}{G(c_i)} \to [C(\lambda^i)] \cong \bigoplus_{j: \lambda^j\in R(\lambda^i)}(\xi^{ij} \colon \frac{\CC_{ij}}{G(c_{ij})}\to [C(\lambda^i)]).$$ 
\end{proposition}
\begin{proof}
	Observe that $X_i\setminus X_{i-1}=\frac{E(\lambda^i)}{G(c_i)}\cong\frac{\CC^{\ell(\lambda^i)}}{G(c_i)}$. Since $T^n$ is abelian, the $T^n$ action on $E(\lambda^i)\cong \CC^{\ell(\lambda^i)}$ determines the following decomposition
	$$E(\lambda^i)\cong \bigoplus_{j: \lambda^j \in R(\lambda^i)}\CC_{ij}$$ for some irreducible representation $\CC_{ij}$ of $T^n$. By \cite[Proposition 2.8]{GGKRW} there exists a finite covering map $q\colon T^n\to T^n$ such that the projection map $\phi\colon {E(\lambda^i)} \to \frac{E(\lambda^i)}{G(c_i)} $ is equivariant via the map $q$ (i.e., $\phi(tx)=q(t)\phi(x)$).
	  Therefore,
	   $$\frac{E(\lambda^i)}{G(c_i)}\cong \bigoplus_{j:\lambda^j\in R(\lambda^i)}\frac{\CC_{ij}}{G(c_{ij})}$$ for some positive integers $c_{ij}$ which divides $c_i$. Hence the proof follows.	
\end{proof}

\begin{remark}\label{rem_eq_strati}
\begin{enumerate}
\item The attaching map $\eta_{i} \colon S(\xi^i) \to X_{i-1}$ for the $q$-cell structure in \eqref{filtration} satisfies $\eta_i|_{S(\xi^{ij})}=f_{ij}.\xi^{ij}$ where $f_{ij}\colon [C(\lambda^i)]\to [C(\lambda^j)]$ is the constant map.

\item The equivariant Euler classes $\{e_{T^n}(\xi^{ij})~|~j<i\}$ are non zero divisors. They are pairwise prime by \cite[Lemma 5.2]{HHH} and the $T^n$-action on $E(\lambda^i)$ discussed in the proof of Theorem \ref{thm:tinv_cell}.
\end{enumerate}
\end{remark}

\begin{theorem}\label{prop_GKM-description}
	Let ${\rm WGr}(d,n)$ be a weighted Grassmann orbifold for $d <n$ corresponding to $W=(w_1,w_2,\dots,w_n)\in(\mathbb{Z}_{\geq0})^n~ \mbox{and}~ a\geq 1$.  
Then the generalized $T^n$-equivariant cohomology $\mathcal{E}^\ast_{T^n}({\rm WGr}(d,n);\QQ)$ 
	 can be given by
	\begin{equation*}
 \Big \{(f_i) \in \bigoplus_{i=0}^{m} \mathcal{E}^*_{T^n}(pt;\QQ) ~ \big{|} ~ e_{T^n}(\xi^{ij })~
		\mbox{divides} ~ f_i - f_j ~\mbox{for }~ j < i ~\mbox{and}~  |\lambda^j\cap \lambda^i|=d-1 \Big\}
	\end{equation*}
for $\mathcal{E}^\ast_{T^n}=K^*_{T^n}$, $H^*_{T^n}$, where $e_{T^n}(\xi^{ij})$ represents the equivariant Euler class of $\xi^{ij}$.
\end{theorem}
\begin{proof}
	 This follows from  \cite[Proposition 2.3]{SaSo} using Proposition \ref{thm:q-cw_srtucture}, \ref{prop:cond_hhh} and Remark \ref{rem_eq_strati}. 
\end{proof}

We note that equivariant cohomology ring of $\WG(d,n)$ with rational coefficients is discussed in \cite{AbMa1}. In the rest, we give a description of the equivariant cohomology ring, equivariant $K$-theory ring and equivariant cobordism ring of a divisive weighted Grassmann orbifold with integer coefficients.

\begin{proposition}\label{thm:cw_srtucture}
	Let ${\rm WGr}(d,n)$ be a divisive weighted Grassmann orbifold for $d < n$ corresponding to $W=(w_1, w_2, \dots, w_n) \in (\mathbb{Z}_{\geq 0})^n$ and $ a\geq 1$. Then there is a $T^n$-invariant stratification 
	\begin{equation*}
		\{pt\} = X_0 \subset X_1 \subset \cdots \subset X_{m}={\rm WGr}(d,n)
	\end{equation*}
	such that the quotient $X_{i}/X_{i-1}$ 
	is homeomorphic to the Thom space $Th(\xi^{i})$
	of the $T^n$-vector bundle 
	\begin{equation*}
		\xi^{i} \colon \CC^{\ell(\lambda^i)} \to [C(\lambda^i)],
	\end{equation*}
	 for $i=1, \dots, m$. 
\end{proposition}
\begin{proof}
Since $\WG(d, n)$ is divisive, there exists $\sigma \in S_n$ such that $\sigma c_i$ divides $\sigma c_{i-1}$ for $i=1,2,\dots,m$. Then $\gcd\{\sigma c_0,\sigma c_1,\dots,\sigma c_i\}=\sigma c_i$ for all $i$. By Theorem \ref{lem_euiv_sn}, one can write $\WG(d,n)=\sqcup_{i=0}^m\frac{\sigma E(\lambda^i)}{G(\sigma c_i)}$. By Lemma \ref{thm_red_loc_gp}, the $q$-cell $\sigma E(\lambda^i)/G(\sigma c_i)$ is homeomorphic to $\sigma E(\lambda^i)/G(\frac{\sigma c_i}{\sigma c_i})\cong \CC^{\ell(\lambda^i)}$ for $i=1, \ldots, m$. Let $X_k=\sqcup_{i=0}^k\frac{\sigma E(\lambda^i)}{G(\sigma c_i)}$ for $i=0,1, \ldots, m$. Remaining follows from the proof of Proposition \ref{thm:q-cw_srtucture}. 	
\end{proof}

\begin{remark}\label{prop_div_cond}
For a divisive weighted Grassmann orbifold, Proposition \ref{prop:cond_hhh} and Remark \ref{rem_eq_strati} hold with $c_{ij}=1$ for every $j<i$.
\end{remark}

\begin{theorem}\label{prop_GKM-description_Z}
		Let ${\rm WGr}(d,n)$ be a divisive weighted Grassmann orbifold for $d <n$.  
Then the generalized $T^n$-equivariant cohomology $\mathcal{E}^\ast_{T^n}({\rm WGr}(d,n); \ZZ)$ 
	 can be given by
	\begin{equation*}
 \Big \{(f_i) \in \bigoplus_{i=0}^{m} \mathcal{E}^*_{T^n}(pt; \ZZ) ~ \big{|} ~ e_{T^n}(\xi^{ij})~
		\mbox{divides} ~ f_i - f_j ~\mbox{for }~ j < i~\mbox{and}~  |\lambda^j\cap \lambda^i|=d-1\Big\}
	\end{equation*}
for $\mathcal{E}^\ast_{T^n}= H^\ast_{T^n}, K^*_{T^n}$ and $MU^*_{T^n}$.
\end{theorem}

\begin{proof}
	This follows from Proposition \ref{thm:cw_srtucture}, Remark \ref{prop_div_cond} and \cite[Theorem 2.3]{HHH}.
\end{proof}

\begin{remark}\label{rem:div_cond}
Let $\lambda^i$ and $\lambda^j$ be two Schubert symbols with $j<i$. If ${\rm WGr} (d,n)$ is a divisive weighted Grassmann orbifold then there exists a permutation $\sigma\in S_n$ such that $\sigma c_i$ divides $\sigma c_j$. We denote $\sigma d_{ij} :=\frac{\sigma c_j}{\sigma c_i} \in \ZZ$. 
\end{remark}

 Next we discuss how to compute $e_{T^n}(\xi^{ij})$. We recall that $$H_{T^n}^*(pt;\ZZ)=H^{*}(BT^n;\ZZ)\cong \ZZ[y_1, y_2, \dots, y_n]$$ where $y_1, y_2, \dots, y_n$ be the standard basis of $H^2(BT^n;\ZZ)$. 
 Using \eqref{eq:char_edge} the character of the one-dimensional representation for the bundle $\xi^{ij}$ is given by
 \begin{equation}\label{eq_one_dim_rep}
 	(t_1,t_2,\dots,t_n)\to t_{\sigma\lambda^{j}}(t_{\sigma\lambda^i})^{-\frac{\sigma c_j}{\sigma c_i}}.
 \end{equation}
 
 Also
 $$K^*_{T^n} (pt) \cong R(T^n)[z, z^{-1}]$$ where $R(T^n)$ is the complex representation ring of $T^n$ and $z$ is the Bott element in $K^{-2}(pt)$. Note that $R(T^n)$ is isomorphic to the ring of Laurent polynomials with $n$-variables, that is $R(T^n) \cong \ZZ[\alpha_1, \ldots, \alpha_n]_{(\alpha_1 \cdots \alpha_n)}$, where  $\alpha_i$ is the irreducible representation corresponding to the projection on the $i$-th factor, see \cite{Hus}. Therefore, using \eqref{eq_one_dim_rep} one has the following.

\begin{align}\label{eq_eu_cl}
e_{T^n}(\xi^{ij}) = 
     \begin{cases}
       1-  \alpha_{\sigma \lambda^{j}} \alpha_{\sigma \lambda^i}^{-\sigma d_{ij}} & \quad\text{in}~  K^0_{T^n} \\
           e_{T^n}(\alpha_{\sigma \lambda^{j}} \alpha_{\sigma\lambda^i}^{-\sigma d_{ij}})  &\quad\text{in}~ MU^2_{T^n}\\
Y_{\sigma \lambda^j} - \sigma d_{ij}Y_{\sigma \lambda^i} & \quad\text{in}~ H^2_{T^n}
     \end{cases}
\end{align}
for $j < i$ and  $|\lambda^j\cap \lambda^i|=d-1$,
where $Y_{\lambda}:=\sum_{i=1}^{d} y_{\lambda_i}$ and $\alpha_{\lambda} = \alpha_{\lambda_1} \cdots \alpha_{\lambda_d}$ for a Schubert symbol $\lambda=(\lambda_1, \ldots, \lambda_d)$.
 
We remark that the structure of  $MU_{T^n}^\ast(pt)$ is unknown, however 
it is referred as the ring of \emph{$T^n$-cobordism forms} in \cite{HHRW}.

\begin{example}
	Consider the weighted Grassmann orbifold $\WG (2,4)$ for $W=(12,2,2,2)$ and $a=6$. We have the ordering on the $6$ Schubert symbols given by 
	$$\lambda^0=(1,2)<\lambda^1=(1,3)<\lambda^2=(1,4)<\lambda^3= (2,3)< \lambda^4=(2,4)<\lambda^5=(3,4).$$
	 Then $c_0=20, c_1=20, c_2=20, c_3=10, c_4=10, c_5=10$ from \eqref{eq_wli}. Here $c_i$ divides $c_{i-1}$ for all $i=1,2,3,4,5$. Thus $\WG(2,4)$ is divisive for the identity permutation in $S_4$. 
	Then $d_{ij}=\frac{c_j}{c_i}$ in Remark \ref{rem:div_cond} gives
	\begin{align*}
		d_{ij} = 
		\begin{cases}
			1 &  \quad\text{if } j<i \text{ and both}~ i,j\in\{0,1,2\} ~\mbox{or},~\{3,4,5\}   \\
			2  &\quad \mbox{if}~ j\in\{0,1,2\}~\mbox{and}~ i\in \{3,4,5\}.			
		\end{cases}
	\end{align*}
Then one can calculate the equivariant Euler class $e_{T^n}(\xi^{ij})$ from \eqref{eq_eu_cl}. The generalized integral equivariant cohomology ring $\mathcal{E}_{T^n}^*(\WG(2,4);\ZZ)$ of this divisive weighted Grassmann orbifold $\WG (2,4)$ can be described by Theorem \ref{prop_GKM-description_Z}. \qed
\end{example}

The following result gives equivariant cohomology ring of some non-divisive weighted Grassmann orbifolds  with integer coefficients.
\begin{theorem}\label{th_eq_coh_non_div}
	Let ${\rm WGr} (d,n)$ be a weighted Grassmann orbifold corresponding to the order $\lambda^0 < \cdots < \lambda^m$ such that $c_i | c_k$ for $k \leq i$ and $i\geq 2$ but $c_1$ does not divide $c_0$. Then the integral equivariant cohomology ring of  ${\rm WGr} (d,n)$ is given by 
	\begin{align*}
		& H_{T^n}^{*}({\rm WGr} (d,n);\mathbb{Z})\\ & =\{(f_i)\in \bigoplus_{i=0}^{m}\mathbb{Z}[y_1,y_2,\dots, y_n]~\big{|}~(Y_{\lambda^j}-d_{ij}Y_{\lambda^i})~\mbox{divides}~(f_i-f_j) \text{ if } j<i, \nonumber \\ & |\lambda^j\cap \lambda^i|=d-1; (i,j)\neq(0,1)~\mbox{and}~ c_1Y_{\lambda^0}-c_0Y_{\lambda^1}~\text{divides}~(f_1-f_0)\}.  
	\end{align*}
\end{theorem}

\begin{proof}
	By the given condition $\gcd\{c_0,c_1,\dots,c_i\}=c_i$ for $i\geq 2$. So, by Lemma \ref{thm_red_loc_gp}, we have $E(\lambda^i)/G(c_i)$ is homeomorphic to $E(\lambda^i)/G(c_i/c_i)\cong \CC^{\ell(\lambda^i)}$  
	for  $i=1, \ldots, m$.
	When $i=1,$ we have $X_1$ is equivariantly homeomorphic to $\WW P(c_0,c_1)$. Therefore it has a $T^n$-invariant cell structure. Thus by \cite[Theorem 2.9]{DKS}, we get the result.
\end{proof}

Next we discuss the equivariant cohomology ring of the weighted projective space $\WW P(b_0,b_1 \ldots, b_m)$ where $(b_0,b_1 \ldots, b_m)\in (\ZZ_{\geq 0})^{m+1}$ for several torus actions. By remark \ref{rmk_wgt_pr_sp}, $\WW P(b_0,b_1 \ldots, b_m)=\WG(1,m+1)$ associated to the weight $W=(b_0-1,\dots,b_m-1)$ and $a=1$. The Schubert symbols for $1<m+1$ are $\{1\},\ldots,\{m\} \text{ and } \{m+1\}$. Assume that $\WG(1,m+1)$ is divisive corresponding to this order (i.e $b_i$ divides $b_{i-1}$, for $i=1,2,\dots,m$). Then $$E(i+1)\cong \{[(u_0,u_1,\dots,u_{i-1},1,0,\dots,0)]\in \WW P(b_0,b_1 \ldots, b_m)\}\cong \CC^i$$
for $i=0,1,\dots,m$.

Let $(n, d)$ be the pair such that $d<n$ and ${n \choose{d}}=m+1$. Then \eqref{T_act} gives a $T^{n}$- action on $\WW P(b_0,b_1 \ldots, b_m)$. Recall $t_{\lambda^i}$ from \eqref{T_act} for the Schubert symbols $\lambda^0, \lambda^1,\dots,\lambda^m$ corresponding to $d<n$. Therefore we have the following,
\begin{align*}
	&(t_1,t_2,\dots ,t_n)[(u_0,u_1,\dots,u_{i-1},1,0,\dots,0)]\\ 
	&=[(t_{\lambda^0} u_0, t_{\lambda^1} u_1, \dots,t_{\lambda^{i-1}} u_{i-1}, {t_{\lambda^i}}, 0,\dots,0)]\\ 
	&=[((t_{\lambda^i})^{-\dfrac{b_0}{b_i}} t_{\lambda^0}u_{0},(t_{\lambda^i})^{-\dfrac{b_1}{b_i}}t_{\lambda^1}u_{1},\dots,(t_{\lambda^i})^{-\dfrac{b_{i-1}}{b_{i}}} t_{\lambda^{i-1}}u_{i-1},1,0,\dots,0)].
\end{align*}

Then $E(i+1)$ is $T^n$-invariant as well as $T^{m+1}$-invariant.
   Let $$X_i:=[(u_0,u_1,\dots,{u_i},0,\dots,0)]\in \WW P(b_0,b_1 \ldots, b_m)\}.$$ Then $X_i$ gives a filtration 
   \begin{equation}\label{eq_fil_wgt_pr}
   \{pt\} = X_0 \subset X_1 \subset \cdots \subset X_{m}=\WW P(b_0,b_1,\dots,b_m).
   \end{equation}
    Note that the filtration in \eqref{eq_fil_wgt_pr}
   satisfies Proposition \ref{thm:cw_srtucture} and Remark \ref{prop_div_cond}. 
   Thus in this case $$\xi^i\colon E(i+1)\to [e_{i+1}]\cong \bigoplus_{j=0}^i(\xi^{ij}:\CC_{ij}\to [e_{i+1}])$$ for some irreducible representation $\CC_{ij}$. Then one can get the following result using the proof of \cite[Theorem 2.3]{HHRW}.

\begin{theorem}\label{thm_eq_coh_pr_sp}
If $\WW P(b_0, \ldots, b_m)$ is divisive, then the generalized $T^{n}$-equivariant cohomology $\mathcal{E}^\ast_{T^{n}}(\WW P(b_0, \ldots, b_m); \ZZ)$ 
	 can be given by
	\begin{equation*}
 \Big \{(f_i) \in \bigoplus_{i=0}^{m} \mathcal{E}^*_{T^{n}}(pt; \ZZ) ~ \big{|} ~ e_{T^{n}}(\xi^{ij})~
		\mbox{divides} ~ f_i - f_j ~\mbox{for all }~ j<i \Big\}
	\end{equation*}
for $\mathcal{E}^\ast_{T^{n}}= H^\ast_{T^{n}}, K^*_{T^{n}}$ and $MU^*_{T^{n}}$.
\end{theorem}

We note that there are several pairs $(n,d)$ such that $d<n$ and ${n\choose d}=m+1>2$. Now we discuss how to calculate the equivariant Euler class $e_{T^{n}}(\xi^{ij})$ in Theorem \ref{thm_eq_coh_pr_sp}. The corresponding one dimensional representation on the bundle $\xi^{ij}$ for $j<i$ is determined by the character 
\begin{equation*}
	(t_1, \ldots, t_n) \to (t_{\lambda^i})^{-\dfrac{b_j}{b_i}}t_{\lambda^j}.
\end{equation*}
Thus, similar to, \eqref{eq_eu_cl} one can calculate the equivariant Euler class $e_{T^n}(\xi^{ij})$ of the bundle $\xi^{ij}$ for $j<i$.

\begin{example}
	For $m=2$, we have ${3\choose 1}={3 \choose 2}=3$. Thus corresponding to two different pairs $(3,1)$ and $(3,2)$ we have two different $T^3$ action on $\WW P(b_0,b_1,b_2)$. The map $f\colon T^3 \to T^3$ defined by $(t_1,t_2,t_3)\to (t_1t_2, t_1t_3, t_2t_3)$ is not an automorphism. So these actions are not equivalent. However, using Theorem \ref{thm_eq_coh_pr_sp}, one can calculate the equivariant cohomology of $\WW P(b_0,b_1,b_2)$ for both the actions if $b_i$ divides $b_{i-1}$ for $i=1,2$. \qed 
\end{example}

\section{Equivariant Schubert calculus for divisive weighted Grassmann orbifolds}\label{sec:schub_calc}
In this section, we show that there exist equivariant Schubert classes which form a basis for the equivariant cohomology ring of a divisive weighted Grassmann orbifold with integer coefficients. Moreover, we compute the weighted structure constants corresponding to this equivariant Schubert basis with integer coefficients.

For $\text{x}\in H_{T^n}^{*}(\WG (d,n);\mathbb{Z})$, the support of $\text{x}$ denoted by `$\rm{supp}(x)$' is the set of all Schubert symbols $\lambda^i$ such that $x|_{\lambda^i} \neq 0$. Recall the partial order `$\preceq$' on the Schubert symbols defined in \eqref{eq_bru_ord}. We follow this partial order `$\preceq$' and we call an element $x \in H_{T^n}^{*}(\WG (d,n);\mathbb{Z})$ is supported above by $\lambda^i$ if $\lambda^i \preceq \lambda^{k}$ for all $\lambda^{k}\in \rm{supp}(x)$.

 Let $\WG(d,n)$ be a divisive weighted Grassmann orbifold. Then there exists $\sigma\in S_n$ such that 
 \begin{equation}\label{div_cond}
  \sigma c_i
  	 ~\text{divides } \sigma c_{i-1}~\text{for}~i=1,2,\dots,m.
\end{equation}

Using Theorem \ref{lem_euiv_sn}, it is sufficient to consider $\sigma=\rm{Id}$ (the identity permutation on $S_n$). For $\sigma=\rm{Id}$, \eqref{div_cond} transforms to $$c_i~\text{divides } c_{i-1} ~\text{for}~i=1,2,\dots,m.$$

 Recall the definition of $R(\lambda^i)$ from \eqref{eq_rev_lbd}. We introduce the following definition. 
   
\begin{definition}\label{def_schub_class}
	An element $x \in H_{T^n}^{*}({\rm WGr} (d,n);\mathbb{Z})$ is said to be  an equivariant Schubert class corresponding to a Schubert symbol $\lambda^i$ if the following conditions are satisfied.
	\begin{enumerate}
		\item  $x|_{\lambda^k}\neq 0  \implies \lambda^i \preceq \lambda^k$ (say that $x$ is supported above $\lambda^i$).
		\item  $x|_{\lambda^i}=\prod_{\lambda^j\in R(\lambda^i)}(Y_{\lambda^j}-\dfrac{c_j}{c_i}Y_{\lambda^i})$.
		\item $x|_{\lambda^k}$ is a homogeneous polynomial of $y_1,y_2,\dots,y_n$ of degree $\ell({\lambda^i})$. 
	\end{enumerate}
\end{definition}

\begin{proposition}[Uniqueness]
	For each Schubert symbol $\lambda^i$, there is at most one equivariant Schubert class $x$ corresponding to $\lambda^i$.
\end{proposition}
\begin{proof}
Suppose that there were two distinct  equivariant Schubert classes $x, x^{'}$ corresponding to $\lambda^i$. Let $\lambda^j$ be the minimal Schubert symbol such that ${(x-x^{'})|_{\lambda^j} \neq 0}$. By Definition \ref{def_schub_class} (1) and (2), we get $\lambda^i \prec \lambda^j$. Then from the condition in the expression of the equivariant cohomology ring in Theorem \ref{prop_GKM-description_Z}, we get that $(x-x^{'})|_{\lambda^j}$ is a multiple of $\prod_{\lambda^k\in R(\lambda^j)}(Y_{\lambda^k}-\dfrac{c_k}{c_j}Y_{\lambda^j})$ which is of degree $\ell(\lambda^j)$. This contradicts the fact that $x-x^{'}$ is homogeneous of degree $\ell(\lambda^i)<\ell(\lambda^j)$. 
\end{proof}

Let us denote the  equivariant Schubert class corresponding to the Schubert symbol $\lambda^i$ by $w\widetilde{S}_{\lambda^i}$ for $i=0,1,\dots,m$. We remark that the existence of $w\widetilde{S}_{\lambda^i}$ follows from \cite[Proposition 4.3]{HHH} and Theorem \ref{prop_GKM-description_Z}. Using the arguments in the proof of \cite[Proposition 1]{KnTa}, one gets the following. 

\begin{proposition}\label{prop_schub_bas}
The equivariant Schubert classes $\{w\widetilde{S}_{\lambda^i}\}_{i=0}^{m}$  is a basis for the module $H^*_{T^n}({\rm WGr} (d,n);\ZZ)$ over $H_{T^n}^{*}(pt; \ZZ)$. Moreover, any $x\in H_{T^n}^{*}({\rm WGr} (d,n);\mathbb{Z})$ can be written uniquely as an $H_{T^n}^{*}(pt; \ZZ)$ linear combination of $w\widetilde{S}_{\lambda^i}$ using only those $\lambda^i$ such that $\lambda^j \preceq \lambda^i$ for some $\lambda^j \in \mbox{supp}(x)$.	 
\end{proposition}

\begin{example}
In Figure \ref{Fig_weig_sch_cls}, we compute the equivariant Schubert class $w\widetilde{S}_{(2,3)}\in H_{T^4}^{*}(\WG (2,4);\ZZ)$ where $\WG (2,4)$ is a divisive weighted Grassmann orbifold for some $W=(\alpha+\gamma\beta,\alpha,\alpha,\alpha)\in (\ZZ_{\geq 0})^4$ and $ a=\beta-2\alpha\in \ZZ_{>0}$. Figure \ref{Fig_weig_sch_cls}(a) is the lattice of the Schubert symbols for $2<4$. Figure \ref{Fig_weig_sch_cls}(b) gives the equivariant Schubert class corresponding to the Schubert symbol $(2,3)$. 
 \begin{figure}
 	\begin{tikzpicture}[scale=.42]
 		
 		\draw (0,0)--(2,-2);
 		\draw (6,0)--(2,-2);
 		\draw [dashed] (0,0)--(4,2);
 		\draw [dashed] (6,0)--(4,2);
 		\draw[dashed] (3.5,5)--(4,2);
 		\draw (3.5,5)--(0,0);
 		\draw (3.5,5)--(2,-2);
 		\draw (3.5,5)--(6,0);
 		\draw[dashed] (2.5,-5)--(4,2);
 		\draw (2.5,-5)--(0,0);
 		\draw (2.5,-5)--(2,-2);
 		\draw (2.5,-5)--(6,0);
 		\node at (3.2,-2.3) {$(1,3)$};
 		\node at (2.5,-5.8) {$(1,2)$};
 		\node at (3, -7) {$(a)$};
 		\node at (-1,0) {$(1,4)$};
 		\node at (7,0) {$(2,3)$};
 		\node at (3,2.1) {$(2,4)$};
 		\node at (3.5,5.5) {$(3,4)$};

 		\begin{scope}[xshift=300]
 				\draw (0,0)--(2,-2);
 			\draw (6,0)--(2,-2);
 			\draw [dashed] (0,0)--(4,2);
 			\draw [dashed] (6,0)--(4,2);
 			\draw[dashed] (3.5,5)--(4,2);
 			\draw (3.5,5)--(0,0);
 			\draw (3.5,5)--(2,-2);
 			\draw (3.5,5)--(6,0);
 			\draw[dashed] (2.5,-5)--(4,2);
 			\draw (2.5,-5)--(0,0);
 			\draw (2.5,-5)--(2,-2);
 			\draw (2.5,-5)--(6,0);
 			\node at (2.5,-2.2) {$0$};
 			\node at (2,-5.3) {$0$};
 			\node at (3,-6.5) {$(b)$};
 			\node at (-0.5,0) {$0$};
 			\node at (12.2,-.1) {$(Y_{(1,3)}-(\gamma+1)Y_{(2,3)})(Y_{(1,2)}-(\gamma+1)Y_{(2,3)})$};
 		\node at (11, 2.5) {$(Y_{(1,4)}-(\gamma+1)Y_{(2,4)})(Y_{(1,2)}-(\gamma+1)Y_{(2,4)})$};
 		\node at (5,5.8) {$(Y_{(1,4)}-(\gamma+1)Y_{(3,4)})(Y_{(1,3)}-(\gamma+1)Y_{(3,4)})$};	
 		\end{scope}
  	\end{tikzpicture}
 	\caption{}
 	\label{Fig_weig_sch_cls}
 \end{figure} \qed
\end{example}

In the rest of this section, we compute the weighted structure constants for the equivariant cohomology of a divisive weighted Grassmann orbifold.
Since the set $\{w\widetilde{S}_{\lambda^i}\}_{i=0}^{m}$ form a $H_{T^n}^{*}(\{pt\};\ZZ)$-basis for $H_{T^n}^{*}(\WG (d,n);\mathbb{Z})$, for any two $\lambda^i$ and $\lambda^j$,  one has the following
\begin{equation}\label{eq_wei_st_con}
w\widetilde{S}_{\lambda^i}~w\widetilde{S}_{\lambda^j}=\sum_{\lambda^k}wc_{ij}^{k}~w\widetilde{S}_{\lambda^k}
\end{equation}
where $\lambda^k\in \{\lambda^0,\lambda^1,\dots,\lambda^m\}$. The constant $wc_{ij}^{k}\in H_{T^n}^{*}(pt; \ZZ) $ in the formula is called `weighted structure constant'.

\begin{lemma}\label{lem_st_con}
   The weighted structure constant $wc_{ij}^{k}$ have the following properties.
	\begin{enumerate}
		\item The weighted structure constant $wc_{ij}^{k}$ has degree $\ell(\lambda^i)+\ell(\lambda^j)-\ell(\lambda^k)$. 
		\item $wc_{ij}^{k}=0$ unless $\ell(\lambda^k)\leq \ell(\lambda^i)+\ell(\lambda^j)$ and $\lambda^k\succeq \lambda^i, \lambda^j$.   
		\item When $i=k$ we have $wc_{ij}^{i}={w\widetilde{S}_{\lambda^j}}|_{\lambda^i}$.
	\end{enumerate}
\end{lemma}
\begin{proof}
	(1) The degree of $w\widetilde{S}_{\lambda^i}$ is $\ell(\lambda^i)$. So the degree of the weighted structure constant $wc_{ij}^{k}$ is given by 
	\begin{align*}
	\text{deg}(wc_{ij}^{k})&=\text{deg}(w\widetilde{S}_{\lambda^i})+\text{deg}(w\widetilde{S}_{\lambda^j})-\text{deg}(w\widetilde{S}_{\lambda^k})\\&=\ell(\lambda^i)+\ell(\lambda^j)-\ell(\lambda^k).
	\end{align*}
	  
	(2) The weighted structure constant $wc_{ij}^{k}=0$ if $\ell(\lambda^i)+\ell(\lambda^j)-\ell(\lambda^k)<0$. Also 
	 $${(w\widetilde{S}_{\lambda^i}~w\widetilde{S}_{\lambda^j})}|_{\lambda^{s}}\neq 0\implies \lambda^{s} \succeq \lambda^i, \lambda^j .$$
	 Thus by Proposition \ref{prop_schub_bas}, $wc_{ij}^{k}\neq 0\implies \lambda^k \succeq \lambda^i, \lambda^j$.
	
	(3) Compare the $\lambda^i$-th component of the both side in  \eqref{eq_wei_st_con} we get $${w\widetilde{S}_{\lambda^i}}|_{\lambda^i}~{w\widetilde{S}_{\lambda^j}}|_{\lambda^i}=wc_{ij}^{i}~{w\widetilde{S}_{\lambda^i}}|_{\lambda^i}+\sum_{k\neq i}wc_{ij}^{k}~{w\widetilde{S}_{\lambda^k}}|_{\lambda^i}.$$
	Since,  $wc_{ij}^{k}=0$ unless $\lambda^k\succeq \lambda^i$. But ${w\widetilde{S}_{\lambda^k}}|_{\lambda^i}=0$ for $\lambda^k\succeq \lambda^i$, and $\lambda^k\neq \lambda^i$. Thus all the terms in the summation vanish. So the claim follows, since ${w\widetilde{S}_{\lambda^i}}|_{\lambda^i}\neq 0$.  
\end{proof}

Now we introduce equivariant Schubert divisor class. Note that $\ell(\lambda^i)=0$ if and only if $i=0$ and $\ell(\lambda^i)=1$ if and only if $i=1$.

\begin{lemma}\label{lem_sch_div_cls}
The equivariant Schubert divisor class $w\widetilde{S}_{\lambda^1}\in H_{T^n}^{*}({\rm WGr} (d,n);\ZZ)$ is given by $${w\widetilde{S}_{\lambda^1}}{|_{\lambda^i}}=Y_{\lambda^0}-\dfrac{c_0}{c_i}Y_{\lambda^i}.$$ 
\end{lemma}
\begin{proof}
Note that ${w\widetilde{S}_{\lambda^1}}|_{\lambda^1} =Y_{\lambda^0}-\dfrac{c_0}{c_1}Y_{\lambda^1}$. For other Schubert symbol $\lambda^i$, it follows from Definition \ref{def_schub_class}. 
\end{proof}

Let  $\lambda^i$ and $\lambda^j$ be two Schubert symbols such that $\lambda^j\leq \lambda^i$. Then Lemma \ref{lem_sch_div_cls} gives
 \begin{equation*}
	{w\widetilde{S}_{\lambda^1}}{|_{\lambda^i}}-{w\widetilde{S}_{\lambda^1}}{|_{\lambda^j}}=\dfrac{c_0}{c_j}(Y_{\lambda^j}-\dfrac{c_j}{c_i}Y_{\lambda^i}).
\end{equation*}

For any two Schubert symbol $\lambda^i$ and $\lambda^j$ we denote $\lambda^i\to\lambda^j$ if $\ell(\lambda^i)=\ell(\lambda^j)+1$ and $\lambda^j\preceq \lambda^i$.

\begin{proposition}[Weighted Pieri rule]\label{prop_pieri_rull}
	$$w\widetilde{S}_{\lambda^1}~w\widetilde{S}_{\lambda^j}=({w\widetilde{S}_{\lambda^1}}|_{\lambda^j})~w\widetilde{S}_{\lambda^j}+\sum_{\lambda^i\to\lambda^j}\dfrac{c_0}{c_j}~w\widetilde{S}_{\lambda^i} .$$
\end{proposition}
\begin{proof}
	Using the fact that deg$(w\widetilde{S}_{\lambda^1})=1$, we have
$$w\widetilde{S}_{\lambda^1}~w\widetilde{S}_{\lambda^j}=(wc^{j}_{1j})~w\widetilde{S}_{\lambda^j}+ \sum_{\lambda^i\to\lambda^j}(wc^{i}_{1j})~w\widetilde{S}_{\lambda^i}. $$
	From Lemma \ref{lem_st_con}, we get $wc^{j}_{1j}={w\widetilde{S}_{\lambda^1}}|_{\lambda^j}$. 
Fix $\lambda^i$ such that $\lambda^i\to\lambda^j$ and compare $\lambda^i$-th component of both side we get
	\begin{align*}
	{w\widetilde{S}_{\lambda^1}}|_{\lambda^i}~{w\widetilde{S}_{\lambda^j}}|_{\lambda^i}&=(wc^{j}_{1j})~{w\widetilde{S}_{\lambda^j}}|_{\lambda^i}+(wc^{i}_{1j})~{w\widetilde{S}_{\lambda^i}}|_{\lambda^i}\\  
	\implies(wc^{i}_{1j})~{w\widetilde{S}_{\lambda^i}}|_{\lambda^i} &=({w\widetilde{S}_{\lambda^1}}|_{\lambda^i}-{w\widetilde{S}_{\lambda^1}}|_{\lambda^j})~{w\widetilde{S}_{\lambda^j}}|_{\lambda^i} \\ 
		\implies (wc^{i}_{1j})~{w\widetilde{S}_{\lambda^i}}|_{\lambda^i} &=\dfrac{c_0}{c_j}~(Y_{\lambda^j}-\dfrac{c_j}{c_i}Y_{\lambda^i})~ {w\widetilde{S}_{\lambda^j}}|_{\lambda^i}.
	\end{align*}
Thus $wc^{i}_{1j}=\dfrac{c_0}{c_j}$, if $\lambda^i\to\lambda^j$. Hence we get the proof.
\end{proof}

 By applying Proposition $\ref{prop_pieri_rull}$ repeatedly we can compute the following as well as the higher products.
 \begin{align*}
 	(w\widetilde{S}_{\lambda^1})^{2}~w\widetilde{S}_{\lambda^j}
 	&=w\widetilde{S}_{\lambda^1}~(({w\widetilde{S}_{\lambda^1}}|_{\lambda^j})~w\widetilde{S}_{\lambda^j}+\sum_{\lambda^i\to\lambda^j}\dfrac{c_0}{c_j}~w\widetilde{S}_{\lambda^i})\\
 	&=({w\widetilde{S}_{\lambda^1}}|_{\lambda^j})^{2}~w\widetilde{S}_{\lambda^j}+\sum_{\lambda^i\to\lambda^j}({w\widetilde{S}_{\lambda^1}}|_{\lambda^j})~\dfrac{c_0}{c_j}~w\widetilde{S}_{\lambda^i}\\
&\quad +\sum_{\lambda^i\to\lambda^j}\dfrac{c_0}{c_j}~({w\widetilde{S}_{\lambda^1}}|_{\lambda^i})~w\widetilde{S}_{\lambda^i}
 +\sum_{\lambda^k\to\lambda^i\to \lambda^j}\dfrac{c_0}{c_j}\dfrac{c_0}{c_i}~w\widetilde{S}_{\lambda^k}.
\end{align*}

\begin{proposition}\label{prop_pieri_rull2}
	For any three Schubert symbols $\lambda^i,\lambda^j \text{ and }\lambda^k$, we have the following recurrence relation $$({w\widetilde{S}_{\lambda^1}}|_{\lambda^k}-{w\widetilde{S}_{\lambda^1}}|_{\lambda^i})wc_{ij}^{k}=(\sum _{\lambda^s\to\lambda^i}\dfrac{c_0}{c_i}wc_{sj}^{k}-\sum _{\lambda^k\to \lambda^t}\dfrac{c_0}{c_t}wc_{ij}^{t}).$$
\end{proposition}
\begin{proof}
	We use the associativity of the multiplication in $H_{T^n}^{*}(\WG (d,n);\ZZ)$ and weighted Pieri rule to expand $w\widetilde{S}_{\lambda^1}w\widetilde{S}_{\lambda^i}w\widetilde{S}_{\lambda^j}$ in two different ways.
 \begin{align}\label{1st_eq}
(w\widetilde{S}_{\lambda^1}w\widetilde{S}_{\lambda^i})w\widetilde{S}_{\lambda^j} 
&=(({w\widetilde{S}_{\lambda^1}}|_{\lambda^i})w\widetilde{S}_{\lambda^i}+\sum_{\lambda^{s}\to \lambda^i}\dfrac{c_0}{c_i}w\widetilde{S}_{\lambda^s})w\widetilde{S}_{\lambda^j}\\
&=({w\widetilde{S}_{\lambda^1}}|_{\lambda^i})\sum_{\lambda^l}wc_{ij}^{l}w\widetilde{S}_{\lambda^l}+\sum_{\lambda^s\to \lambda^i}\dfrac{c_0}{c_i}\sum_{\lambda^l}wc_{sj}^{l}w\widetilde{S}_{\lambda^l}.\nonumber
\end{align}	

 \begin{align}\label{2nd_eq}
w\widetilde{S}_{\lambda^1}(w\widetilde{S}_{\lambda^i}w\widetilde{S}_{\lambda^j}) &=w\widetilde{S}_{\lambda^1}\sum_{\lambda^l}wc_{ij}^{l}w\widetilde{S}_{\lambda^l}\\
&=\sum_{\lambda^l}wc_{ij}^{l}(({w\widetilde{S}_{\lambda^1}}|_{\lambda^l})w\widetilde{S}_{\lambda^l}+\sum_{\lambda^r\to \lambda^l}\dfrac{c_0}{c_l}w\widetilde{S}_{\lambda^r}).\nonumber
\end{align}	
	Comparing the coefficient of $w\widetilde{S}_{\lambda^k}$ in \eqref{1st_eq} and \eqref{2nd_eq} we get
	$$({w\widetilde{S}_{\lambda^1}}|_{\lambda^i})wc_{ij}^{k}+\sum _{\lambda^s\to\lambda^i}\dfrac{c_0}{c_i}wc_{sj}^{k}=wc_{ij}^{k}({w\widetilde{S}_{\lambda^1}}|_{\lambda^k})+\sum _{\lambda^k\to \lambda^t}\dfrac{c_0}{c_t}wc_{ij}^{t}.$$
\end{proof}

\vspace{1cm} 

 \noindent  {\bf Acknowledgment.} 
 The first author thanks `Indian Institute of Technology Madras' for PhD fellowship. The second author thanks `Indian Institute of Technology Madras' for SEED research grant.

\bibliographystyle{amsalpha}
\bibliography{ref-Wedge1}

\end{document}